\documentclass{article}
\usepackage[T1]{fontenc}
\usepackage[utf8]{inputenc}
\usepackage[english]{babel}
\usepackage{graphicx}
\usepackage{amsmath}
\usepackage[all]{xy} 
\usepackage{lmodern}
\usepackage[pagebackref]{hyperref}
\usepackage{grffile}
\usepackage{adjustbox}
\usepackage{color}
\usepackage{xcolor}
\usepackage{amsthm} 
\usepackage{subcaption}
\usepackage{amssymb}
\usepackage{stmaryrd}
\usepackage{caption}
\usepackage{url}
\usepackage{faktor}
\usepackage{mathrsfs}
\usepackage{relsize,exscale}
\usepackage{multicol}
\usepackage{mathabx}
\usepackage{float}
\usepackage{amsbsy}
\usepackage{appendix}

\usepackage{xfrac}  
\usepackage{changebar}
\usepackage{vertbars}
\usepackage{pgf,tikz}
\usetikzlibrary{arrows}
\usetikzlibrary{decorations}
\usetikzlibrary{patterns}
\usepackage{comment}
\usepackage{wasysym}
\usepackage{titlesec}
\usepackage{bbold}
\usepackage{pgfplots,pgfplotstable}
\usepgfplotslibrary{patchplots}
\usepackage[rightcaption]{sidecap}
\usepackage[a4paper, total={6in, 9in}]{geometry}
\usepackage{thmtools, thm-restate}
\usepackage{algorithm}
\usepackage{algpseudocode}
\usepackage{ulem}
\usepackage{nicematrix}
\usepackage{pict2e}
\usepackage[nameinlink,capitalize]{cleveref}
\colorlet{perso}{blue!50!cyan}
\definecolor{amber}{rgb}{1.0, 0.75, 0.0}

\newtheorem{Definition}{Definition}[section]

\newtheorem{Theoreme}[Definition]{Theorem}
\newtheorem{Proposition}[Definition]{Proposition}
\newtheorem{Notation}[Definition]{Notation}

\newtheorem{Lemma}[Definition]{Lemma}
\newtheorem{Remarque}[Definition]{Remark}
\crefname{figure}{fig.}{Fig.}
\crefname{section}{sec.}{Sec.}

\newcommand{\ZZ}{\mathbb{Z}}

\newcommand{\RR}{\mathbb{R}}
\newcommand{\NNN}{\mathbb{Z}_{>0}}

\newcommand{\NN}{\mathbb{N}}
\newcommand{\entn}{\{1,\ldots,n\}}
\newcommand{\entk}{\{1,\ldots,\kappa\}}
\def \ent#1#2{\{#1,\ldots,#2\}}

\newcommand{\zerun}{[0,1]}

\newcommand*\Bell{\ensuremath{\boldsymbol\ell}}
\newcommand*\Bcp{\ensuremath{\boldsymbol\cp}}

\DeclareMathOperator{\AP}{AP}

\DeclareMathOperator{\cotan}{cotan}

\newcommand{\Leb}{\mathsf{Leb}}
\newcommand{\ST}{{\sf SubTan}_{K}}
\newcommand{\zn}{\mathbf{z}[n]}
\newcommand{\bz}{\mathbf{z}}
\newcommand{\zzn}{{z}[n]}
\newcommand{\Aff}{\mathsf{Aff}}
\newcommand{\wj}{j+1}

\newcommand{\CH}{{\sf CH}}

\newcommand{\ie}{\textit{i.e. }}
\newcommand{\cvg}{\underset{n\to\infty}{\longrightarrow}}

\def \sur#1#2{\mathrel{\mathop{\kern 0pt#1}\limits^{#2}}}
\def \norm#1{\left\lVert #1 \right\rVert}
\def \indic#1{\mathbb{1}_{\left\{#1\right\}}}
\def \flr#1{\left\lfloor#1\right\rfloor}
\def \abso#1{\left\vert#1\right\vert}

\def \proba{\xrightarrow[n]{(proba.)}}

\def \dd{\xrightarrow[n]{(d)}}

\def \Area#1{{\sf Area}\left(#1\right)}
\def \P{\mathbb{P}}
\def \Dom#1{{\sf Dom}(#1)}
\def \Tang#1{{\sf Tangency}_{\kappa}(#1)}

\newcommand{\Ck}{\mathfrak{C}_\kappa}
\newcommand{\ECP}{\mathsf{ECP}}
\newcommand{\PCP}{\mathsf{PCP}}
\newcommand{\bb}{\mathsf{b}}
\newcommand{\corner}{\mathsf{corner}}

\newcommand{\sk}{\mathbf{s}^{(n)}[\kappa]}
\newcommand{\ssk}{s[\kappa]}
\newcommand{\Nkn}{\mathbb{N}_\kappa(n)}
\newcommand{\bNkn}{\widebar{\mathbb{N}_{\kappa-1}}(n)}
\newcommand{\lk}{\Bell^{(n)}[\kappa]}

\def \cp{{\sf cp}}
\newcommand{\Lj}{\ell[\kappa]}
\newcommand{\Cj}{c[\kappa]}
\newcommand{\ve}{{\sf v}}
\newcommand{\pp}{{\sf p}}
\newcommand{\sth}{\sin(\theta_\kappa)}

\newcommand{\cl}{\mathfrak{cl}}

\newcommand{\pk}{\mathbb{P}_\kappa(n)}

\newcommand{\ptK}{\widetilde{\mathbb{P}}_K(n)}

\newcommand{\QtnK}{{\mathsf{D}}^{(n)}_{K}}
\newcommand{\UnK}{\mathsf{U}^{(n)}_{K}}
\newcommand{\tst}{\text{ such that }}

\newcommand{\pK}{\mathbb{P}_K(n)}
\newcommand{\CVnK}{{\mathcal{C}_K}(n)}
\newcommand{\Faktor}{{\sf Faktor}}
\newcommand{\snj}{\mathbf{s}^{(n)}_{j}}
\newcommand{\polyk}{\mathbf{P}_\kappa}
\newcommand{\polyktau}{\mathbf{P}_{\kappa}^{\mathcal{T}}}
\newcommand{\polytau}{\mathbf{P}^{\mathcal{T}}}
\newcommand{\Ktau}{K_{\mathcal{T}}}

\begin{document}

	\author{Ludovic Morin}
	\date{\footnotesize Univ. Bordeaux, CNRS, Bordeaux INP, LaBRI, UMR 5800, F-33400 Talence, France}
	\title{Probability that $n$ points are in convex position in a general convex polygon :\\ Asymptotic results.}
	\maketitle

	\subsection*{Abstract}

	Let $\pK$ be the probability that $n$ points $z_1,\ldots,z_n$ picked uniformly and independently in $K$, a compact convex polygon in $\RR^2$ with non-empty interior, are in convex position, that is, form the vertex set of a convex polygon.
	In this paper, we determine the exact asymptotic behavior of $\pK$ as $n\to+\infty$. This improves on a famous result of Bárány \cite{barany2} (yet valid for a general convex set $K$) and a result initiated in the case where $K$ is a regular convex polygon \cite{Morin}.

	\section{Introduction}
	
	All along the paper, for any integer $\kappa$, $x[\kappa]$ will stand for $(x_1,\ldots,x_\kappa)$ whatever the type of the $x_i$.\par
	
	For any compact convex set $K$ in $\RR^2$ with non-empty interior and for any $n\in\NN$, we let $\mathsf{U}_K^{(n)}$ denote the law of an $n$-tuple $\zn:=(\bz_1,\cdots,\bz_n)$, where the $\bz_i$ are independent and identically distributed (i.i.d.) and uniform in $K$. We let $\mathbb{P}_K(n)$ denote the probability that $\{\bz_1,\cdots,\bz_n\}$ forms the vertex set of a convex polygon.\par
	In the paper, we will mainly be interested in establishing new asymptotic results for $\P_K(n)$ in the case where $K$ is a convex polygon.
	It will be convenient to fix a canonical description of a polygon $K$, and to suppose that $K$ is placed somewhere precisely in the plane. This does not alter the degree of generality of our results, since any inversible affine map $\Aff$ preserves convexity and the uniform distribution, so that $\P_K(n)=\P_{\Aff(K)}(n)$.\\
	
	\paragraph{Notation.} In the sequel, $\kappa\geq3$ is considered to be fixed. We will work quite a lot with indices $j$ running through the set of integers $\entk$. By convention, in the case $j=1$, $j-1$ stands for $\kappa$, and when $j=\kappa$, $j+1$ stands for 1 (we do so to avoid tedious notation).
	
	We let then $\polyk$ be the set of polygons $K$ with non-empty interior having $\kappa$ vertices (vertices are extremal points here), included in $\RR^+\times \RR$, avoiding $(-\infty,0)\times\{0\}$,  having a vertex at $(0,0)$, and a side included in $[0,+\infty)\times \{0\}$, as on \Cref{chap3:fig1}. In addition, let $\mathbf{P}$ be the union on $\kappa$ of all $\polyk$ (the set of all polygons with non-empty interior).\par
	
	We now define ``canonical values'' associated with a polygon $K$ in ${\bf P}_\kappa$ (these values are summarized in \Cref{chap3:fig1}). The vertices $(\ve_1,\cdots,\ve_{\kappa})$ of $K$ are fixed as follows:
	we set $\ve_{1}=(0,0)$, $\ve_2\in[0,+\infty)\times\{0\}$ is the vertices on the horizontal line, and the other vertices are taken counterclockwise around $K$ as on Fig.1.
	We let $r_1,\ldots,r_\kappa$ be the successive lengths of its sides, so that the length $\norm{\ve_j-\ve_{\wj}}$ is $r_j$, for all $j\in\entk$ (where $\norm{x-y}$ is the Euclidean distance between $x$ and $y$), and let $\theta_j$ be the internal angle between the $j^{th}$ side (this of side-length $r_j$) and the $\wj^{th}$ one.\par
	
	\begin{figure}[hbtp]
		\centering
		\begin{tikzpicture}[scale=0.8]

\definecolor{ForestGreen}{RGB}{34,139,34}
\definecolor{darkgray176}{RGB}{176,176,176}
\definecolor{amber}{rgb}{.8, 0.6, 0.4}

\begin{axis}[axis lines=none,
tick align=outside,
tick pos=left,
x grid style={darkgray176},
xmin=-2.5, xmax=4.5,
xtick style={color=black},
y grid style={darkgray176},
ymin=-0.4, ymax=5.5,
ytick style={color=black}
]
\addplot [->,semithick, black]
table {%
0 -0.5
0 5.1
};
\addplot [->,semithick, black]
table {%
-0.3 0
4.1 0
};
\addplot [thick, perso]
table {%
0 0
3 0
4 3
1 5
-2 4
-1 1
0 0
};
\node[] (A) at (axis cs:4.1,0.){};
\draw[above,color=black](A) node { $x$};

\node (K)at (axis cs:0,5.1){};
\draw[left,color=black](K) node { $y$};

\node[inner sep=1.5pt,circle,draw=black,fill=red] (A1) at (axis cs:0.,0.){};
\draw[below left,color=red](A1) node {$\ve_1$};
\node[inner sep=1.5pt,circle,draw=black,fill=red] (A2) at (axis cs:3.,0.){};
\draw[below,color=red](A2) node { $\ve_2$};
\node[inner sep=1.5pt,circle,draw=black,fill=red] (A3) at (axis cs:4.,3.){};
\node[inner sep=1.5pt,circle,draw=black,fill=red] (A4) at (axis cs:1.,5.){};
\draw[above right,color=red](A4) node {$\ve_j$};
\node[inner sep=1.5pt,circle,draw=black,fill=red] (A5) at (axis cs:-2.,4.){};
\node[inner sep=1.5pt,circle,draw=black,fill=red] (A6) at (axis cs:-1.,1.){};
\draw[below,color=red](A6) node {$\ve_\kappa$};

\node[] (R1) at (axis cs:1.5,0.){};
\draw[above,color=perso](R1) node {$r_1$};
\node[] (R2) at (axis cs:3.5,1.5){};
\draw[left,color=perso](R2) node {$r_2$};
\node[] (RJ) at (axis cs:-0.5,4.5){};
\draw[above,color=perso](RJ) node {$r_j$};
\node[] (RK) at (axis cs:-0.375,0.4){};
\draw[above,color=perso](RK) node {$r_\kappa$};

\draw [thick,color=ForestGreen] (axis cs:2.8,0) arc [radius=0.2cm,start angle=180,end angle=70];
\node[] (T1) at (axis cs:3.05,0.05){};
\draw[above left,color=ForestGreen](T1) node {$\theta_1$};

\draw [thick,color=ForestGreen] (axis cs:-1.94,3.81) arc [radius=0.2cm,start angle=-75,end angle=15];
\node[] (TJ) at (axis cs:-1.94,3.81){};
\draw[right,color=ForestGreen](TJ) node {$\theta_j$};

\draw [thick,color=ForestGreen] (axis cs:0.2,0) arc [radius=0.2cm,start angle=0,end angle=135];
\node[] (TK) at (axis cs:0.,0.){};
\draw[above right,color=ForestGreen](TK) node {$\theta_\kappa$};

\end{axis}

\end{tikzpicture}
		\caption{An example of $K\in\mathbf{P}_6$}
		\label{chap3:fig1}
	\end{figure}
	
	There already are asymptotic results for $\P_K(n)$ in the case where $K$ is a general convex set. One of the most important for our work is B\'ar\'any's, who gave in \cite{barany2} a logarithmic equivalent of $\mathbb{P}_K(n)$:
	\begin{Theoreme}\cite{barany2}\label{thm2}
		For any compact convex set $K$ with non empty interior,
		\[\lim_{n\to+\infty} n^2\left(\mathbb{P}_{K}(n)\right)^{\frac{1}{n}}=\frac{1}{4}e^2\frac{\AP^*(K)^3}{\Area{K}},\]
		where 
		\begin{align}\label{eq:APBara}
			\AP^*(K):=\max_{\substack{S \text{convex set} \\ S\subset K}} \AP(S),
		\end{align}
		and $\AP(S)$ denotes the affine perimeter of a convex set $S$.
	\end{Theoreme}

	There are several equivalent ways to define the affine perimeter of a convex set $S$ (\cite{blaschkedifferential,leichtweiss1986,lutwak,schutt1993}). We borrow to Blaschke (\cite{blaschkedifferential}) the following one:
	\begin{Definition}\label{chap3:def1}
		Given $S$ a convex compact set (nonflat), let $x_1,\ldots,x_m,x_{m+1}=x_1$ a subdivision of
		the boundary $\partial S$ and let $d_i$ be the line supporting $S$ at $x_i$ for all $i\in\{1,\ldots,m\}$.
		Write $y_i$ for the intersection of $d_i$ and $d_{i+1}$ (if $d_{i}=d_{i+1}$ then $y_i$ can be any point between
		$x_i$ and $x_{i+1}$). Let $T_i$ denote the triangle with vertices $x_i,y_i,x_{i+1}$ and also its area. We define the affine perimeter of the convex set $S$ as
		\[\AP(S)=2\lim \sum_{i=1}^m\sqrt[3]{T_i}\]
		where the limit is taken over all sequences of subdivisions $x[m]$ with $\max_{1,\ldots,m}\vert x_i-x_{i+1}\vert \to 0$ and $m\to+\infty$.
	\end{Definition}

	B\'ar\'any also proved in \cite{barany1} that there exists a unique convex set $\Dom{K}\subset K$ such that $\AP^*(K)=\AP(\Dom{K})$, and that the boundary of $\Dom{K}$ is either composed of pieces of the boundary of $K$ or of parabola arcs, but does not contains any line segment. Another remarkable characterization of the set $\Dom{K}$ is the fact that for an $n$-tuple $\zn$ under $\mathsf{U}^{(n)}_K$, conditioned to be in convex position, the convex hull $\CH(\zn)$ of the points $\zn$ converges in probability to $\Dom{K}$ for the Hausdorff distance (this is B\'ar\'any's limit shape theorem, see \Cref{thm:lst}). We give an example in \Cref{figintro} in the equilateral triangle as well as the regular octagon.
	
	\begin{figure}[H]
		\hspace{0.5cm}
		\begin{subfigure}[hbtp]{0.45\textwidth}
			\centering
			\input{Figure_triangle.txt}
			
		\end{subfigure}\hspace{1cm}
		\begin{subfigure}[hbtp]{0.45\textwidth}
			\centering
				\begin{tikzpicture}[scale=3.5]

	\def\R{0.6}
	\def\K{8}
	\def\alphaK{180/\K}
	
	\draw[perso,thick](0,0)--(0.455090,0.000000)--(0.776887,0.321797)--(0.776887,0.776887)--(0.455090,1.098684)--(0.000000,1.098684)--(-0.321797,0.776887)--(-0.321797,0.321797)--(0,0);
	
	\node (a1) at (0,0) {};
	\node (a2) at (0.455090,0.000000) {};
	\node (a3) at (0.776887,0.321797) {};
	\node (a4) at (0.776887,0.776887) {};
	\node (a5) at (0.455090,1.098684) {};
	\node (a6) at (0.000000,1.098684) {};
	\node (a7) at (-0.321797,0.776887) {};
	\node (a8) at (-0.321797,0.321797) {};
	
	\foreach \x [remember=\x as \y (initially \K)] in {1,2,...,\K} { 
	\path[draw=perso] (a\x) -- (a\y) node[midway, circle, fill=white, inner sep = 0.0001pt] (b\x) {}; 
	}
	\foreach \x [remember=\x as \y (initially \K)] in {1,2,...,\K} { 
	\path[draw=white,dashed] (b\x) -- (b\y) node[midway] (c\x) {}; 
	}
	\foreach \i in {1,2,...,\K} {
	\begin{scope}[shift={(c\i)},rotate=67.5+360*\i/\K] 
	\draw[thick,red] plot[ domain=-{cos(180/\K)*sin(180/\K)}:{cos(180/\K)*sin(180/\K)}] ({\x*\R},{-\R*\x*\x/2/cos(\alphaK)/cos(\alphaK)+\R*sin(\alphaK)*sin(\alphaK)/2}); 
	\end{scope}
	}

	\node[inner sep=1.pt,circle,draw=black,fill=amber] at (0.185235,0.026023){};

	\node[inner sep=1.pt,circle,draw=black,fill=amber] at (0.215382,0.026242){};

	\node[inner sep=1.pt,circle,draw=black,fill=amber] at (0.312429,0.036071){};

	\node[inner sep=1.pt,circle,draw=black,fill=amber] at (0.348603,0.042879){};

	\node[inner sep=1.pt,circle,draw=black,fill=amber] at (0.360606,0.045377){};

	\node[inner sep=1.pt,circle,draw=black,fill=amber] at (0.389819,0.055116){};

	\node[inner sep=1.pt,circle,draw=black,fill=amber] at (0.423449,0.067254){};

	\node[inner sep=1.pt,circle,draw=black,fill=amber] at (0.426778,0.068861){};

	\node[inner sep=1.pt,circle,draw=black,fill=amber] at (0.456322,0.083486){};

	\node[inner sep=1.pt,circle,draw=black,fill=amber] at (0.480554,0.096739){};

	\node[inner sep=1.pt,circle,draw=black,fill=amber] at (0.482531,0.097993){};

	\node[inner sep=1.pt,circle,draw=black,fill=amber] at (0.528678,0.129749){};

	\node[inner sep=1.pt,circle,draw=black,fill=amber] at (0.546491,0.144448){};

	\node[inner sep=1.pt,circle,draw=black,fill=amber] at (0.560064,0.156033){};

	\node[inner sep=1.pt,circle,draw=black,fill=amber] at (0.584984,0.178691){};

	\node[inner sep=1.pt,circle,draw=black,fill=amber] at (0.603752,0.197028){};

	\node[inner sep=1.pt,circle,draw=black,fill=amber] at (0.607743,0.201806){};

	\node[inner sep=1.pt,circle,draw=black,fill=amber] at (0.615440,0.211614){};

	\node[inner sep=1.pt,circle,draw=black,fill=amber] at (0.624306,0.226221){};

	\node[inner sep=1.pt,circle,draw=black,fill=amber] at (0.633223,0.243742){};

	\node[inner sep=1.pt,circle,draw=black,fill=amber] at (0.637267,0.254927){};

	\node[inner sep=1.pt,circle,draw=black,fill=amber] at (0.653168,0.299712){};

	\node[inner sep=1.pt,circle,draw=black,fill=amber] at (0.661291,0.323367){};

	\node[inner sep=1.pt,circle,draw=black,fill=amber] at (0.706741,0.473266){};

	\node[inner sep=1.pt,circle,draw=black,fill=amber] at (0.713204,0.498770){};

	\node[inner sep=1.pt,circle,draw=black,fill=amber] at (0.722643,0.551044){};

	\node[inner sep=1.pt,circle,draw=black,fill=amber] at (0.723156,0.567010){};

	\node[inner sep=1.pt,circle,draw=black,fill=amber] at (0.724735,0.631203){};

	\node[inner sep=1.pt,circle,draw=black,fill=amber] at (0.722834,0.686894){};

	\node[inner sep=1.pt,circle,draw=black,fill=amber] at (0.716985,0.731149){};

	\node[inner sep=1.pt,circle,draw=black,fill=amber] at (0.716231,0.735356){};

	\node[inner sep=1.pt,circle,draw=black,fill=amber] at (0.713311,0.746364){};

	\node[inner sep=1.pt,circle,draw=black,fill=amber] at (0.700969,0.790832){};

	\node[inner sep=1.pt,circle,draw=black,fill=amber] at (0.690285,0.813382){};

	\node[inner sep=1.pt,circle,draw=black,fill=amber] at (0.682610,0.828534){};

	\node[inner sep=1.pt,circle,draw=black,fill=amber] at (0.676009,0.841297){};

	\node[inner sep=1.pt,circle,draw=black,fill=amber] at (0.661734,0.868003){};

	\node[inner sep=1.pt,circle,draw=black,fill=amber] at (0.651842,0.886362){};

	\node[inner sep=1.pt,circle,draw=black,fill=amber] at (0.647053,0.894465){};

	\node[inner sep=1.pt,circle,draw=black,fill=amber] at (0.638204,0.907258){};

	\node[inner sep=1.pt,circle,draw=black,fill=amber] at (0.627990,0.921733){};

	\node[inner sep=1.pt,circle,draw=black,fill=amber] at (0.618557,0.933120){};

	\node[inner sep=1.pt,circle,draw=black,fill=amber] at (0.610286,0.940988){};

	\node[inner sep=1.pt,circle,draw=black,fill=amber] at (0.587608,0.961299){};

	\node[inner sep=1.pt,circle,draw=black,fill=amber] at (0.573911,0.970351){};

	\node[inner sep=1.pt,circle,draw=black,fill=amber] at (0.569754,0.972555){};

	\node[inner sep=1.pt,circle,draw=black,fill=amber] at (0.508485,0.999161){};

	\node[inner sep=1.pt,circle,draw=black,fill=amber] at (0.489119,1.007557){};

	\node[inner sep=1.pt,circle,draw=black,fill=amber] at (0.478645,1.011635){};

	\node[inner sep=1.pt,circle,draw=black,fill=amber] at (0.468749,1.015455){};

	\node[inner sep=1.pt,circle,draw=black,fill=amber] at (0.429814,1.029485){};

	\node[inner sep=1.pt,circle,draw=black,fill=amber] at (0.393496,1.041680){};

	\node[inner sep=1.pt,circle,draw=black,fill=amber] at (0.355906,1.044602){};

	\node[inner sep=1.pt,circle,draw=black,fill=amber] at (0.330488,1.045541){};

	\node[inner sep=1.pt,circle,draw=black,fill=amber] at (0.319714,1.045881){};

	\node[inner sep=1.pt,circle,draw=black,fill=amber] at (0.261982,1.046579){};

	\node[inner sep=1.pt,circle,draw=black,fill=amber] at (0.215640,1.046618){};

	\node[inner sep=1.pt,circle,draw=black,fill=amber] at (0.128715,1.045220){};

	\node[inner sep=1.pt,circle,draw=black,fill=amber] at (0.093765,1.038423){};

	\node[inner sep=1.pt,circle,draw=black,fill=amber] at (0.090848,1.037812){};

	\node[inner sep=1.pt,circle,draw=black,fill=amber] at (-0.009672,1.011159){};

	\node[inner sep=1.pt,circle,draw=black,fill=amber] at (-0.070295,0.994479){};

	\node[inner sep=1.pt,circle,draw=black,fill=amber] at (-0.076598,0.991011){};

	\node[inner sep=1.pt,circle,draw=black,fill=amber] at (-0.088502,0.983972){};

	\node[inner sep=1.pt,circle,draw=black,fill=amber] at (-0.130887,0.946407){};

	\node[inner sep=1.pt,circle,draw=black,fill=amber] at (-0.145507,0.931454){};

	\node[inner sep=1.pt,circle,draw=black,fill=amber] at (-0.152673,0.923103){};

	\node[inner sep=1.pt,circle,draw=black,fill=amber] at (-0.177316,0.888144){};

	\node[inner sep=1.pt,circle,draw=black,fill=amber] at (-0.197980,0.855754){};

	\node[inner sep=1.pt,circle,draw=black,fill=amber] at (-0.222945,0.809634){};

	\node[inner sep=1.pt,circle,draw=black,fill=amber] at (-0.252082,0.746207){};

	\node[inner sep=1.pt,circle,draw=black,fill=amber] at (-0.270055,0.706679){};

	\node[inner sep=1.pt,circle,draw=black,fill=amber] at (-0.283414,0.664128){};

	\node[inner sep=1.pt,circle,draw=black,fill=amber] at (-0.289899,0.639809){};

	\node[inner sep=1.pt,circle,draw=black,fill=amber] at (-0.300629,0.598673){};

	\node[inner sep=1.pt,circle,draw=black,fill=amber] at (-0.305637,0.577477){};

	\node[inner sep=1.pt,circle,draw=black,fill=amber] at (-0.308226,0.565146){};

	\node[inner sep=1.pt,circle,draw=black,fill=amber] at (-0.315792,0.491111){};

	\node[inner sep=1.pt,circle,draw=black,fill=amber] at (-0.315073,0.458992){};

	\node[inner sep=1.pt,circle,draw=black,fill=amber] at (-0.314486,0.434732){};

	\node[inner sep=1.pt,circle,draw=black,fill=amber] at (-0.312324,0.409753){};

	\node[inner sep=1.pt,circle,draw=black,fill=amber] at (-0.309431,0.393841){};

	\node[inner sep=1.pt,circle,draw=black,fill=amber] at (-0.275097,0.326129){};

	\node[inner sep=1.pt,circle,draw=black,fill=amber] at (-0.254384,0.300403){};

	\node[inner sep=1.pt,circle,draw=black,fill=amber] at (-0.218229,0.265111){};

	\node[inner sep=1.pt,circle,draw=black,fill=amber] at (-0.169148,0.219282){};

	\node[inner sep=1.pt,circle,draw=black,fill=amber] at (-0.162499,0.213225){};

	\node[inner sep=1.pt,circle,draw=black,fill=amber] at (-0.124493,0.180507){};

	\node[inner sep=1.pt,circle,draw=black,fill=amber] at (-0.089456,0.151158){};

	\node[inner sep=1.pt,circle,draw=black,fill=amber] at (-0.049648,0.118759){};

	\node[inner sep=1.pt,circle,draw=black,fill=amber] at (-0.022498,0.102016){};

	\node[inner sep=1.pt,circle,draw=black,fill=amber] at (-0.008110,0.093678){};

	\node[inner sep=1.pt,circle,draw=black,fill=amber] at (0.049650,0.060516){};

	\node[inner sep=1.pt,circle,draw=black,fill=amber] at (0.064062,0.052619){};

	\node[inner sep=1.pt,circle,draw=black,fill=amber] at (0.074831,0.046722){};

	\node[inner sep=1.pt,circle,draw=black,fill=amber] at (0.087917,0.041431){};

	\node[inner sep=1.pt,circle,draw=black,fill=amber] at (0.110182,0.032493){};

	\node[inner sep=1.pt,circle,draw=black,fill=amber] at (0.124080,0.030807){};

	\node[inner sep=1.pt,circle,draw=black,fill=amber] at (0.151104,0.028249){};

	\node[inner sep=1.pt,circle,draw=black,fill=amber] at (0.159422,0.027479){};

\end{tikzpicture}
		\end{subfigure}
		
		\caption{In the case of an equilateral triangle $(n=200)$ and of an octagon $(n=100)$, two samples of $n$ i.i.d. uniform points conditioned to be in convex position. We can see that each of them is very close to a red curve representing the boundary of $\Dom{\triangle}$ and $\Dom{\octagon}$, respectively.}\label{figintro}
	\end{figure}
	
	One of the aims of this paper is to provide, in the case of convex polygons \ie when $K\in\mathbf{P}$, new characterizations of the quantities $\Dom{K}$ and $\AP^*(K)$, which will be needed to provide an actual asymptotic equivalent of the sequence $\left(\mathbb{P}_K(n)\right)_n$.
	
	Since the boundary of $\Dom{K}$ does not contain any line segment, we deduce that if $K$ is a polygon, $\Dom{K}$ has a boundary composed of finitely many parabola arcs. For a convex polygon $K\in\polyk,$ we let the map $\Tang{K}$ denote the set of sides of $K$ to which $\Dom{K}$ is tangent, which is thus a subset of $\entk$. We will prove in \Cref{lem:tangency} that $\abso{\Tang{K}}\geq3$.
	We give an example of two polygons in \Cref{chap3:fig3}.

	\begin{figure}[hbtp]
		\centering
		\hspace{1cm}
		
		\begin{minipage}{0.45\textwidth}
			\centering
			\input{Figure_3.txt}
		\end{minipage}
		\hspace{1ex}
		\begin{minipage}{0.45\textwidth}
			\centering
			\input{Figure_4.txt}
		\end{minipage}
		
		\caption{Two examples of convex polygons drawn in blue, and their limit shape drawn in red. On the left, $K_L$ is in $\mathbf{P}_5^{\mathcal{T}}$, for its limit shape is tangent to every side. This is not the case anymore on the right: $K_R\in \mathbf{P}_7$, but ${\sf Tangency}_7(K_R)=\{1,2,3,5,6\}$ so that $K_R\notin\mathbf{P}_7^{\mathcal{T}}$.}
		\label{chap3:fig3}
	\end{figure}
	\noindent We let $\polyktau$ be the subset of polygons $K\in\polyk$ such that $\Tang{K}=\entk$, \ie $\Dom{K}$ is tangent to all the sides of $K$. The set $\polyktau$ is not empty since it contains $\Ck$, the regular convex $\kappa$-gon of area 1 (and all its images through affine maps).
	Denote by $\polytau$ the set 
	\[\polytau:=\bigcup_{\kappa\geq 3} \polyktau.\]
	
	\subsection{The analysis is simpler when $K\in\polytau$:}
	
	Fix $\kappa\geq3$ and $K\in\polyktau$.
	Since $\Dom{K}$ is tangent to every side of $K$, the affine perimeter of $\Dom{K}$ may be expressed easily in terms of the geometric properties of $K$. Denote by $\pp_j$ the point of tangency of $\Dom{K}$ with the $j^{th}$ side of $K$ (a recap is offered in \Cref{chap3:fig}), and set \begin{align}\label{eq:defwj}
		w_j:=\norm{\ve_j-\pp_j}/r_j \text{ for all } j\in\entk.
	\end{align} The affine perimeter of $\Dom{K}$ is, by \Cref{chap3:def1}, given by 
	\begin{align}\label{eq:AP1}
		\AP(\Dom{K})=2\sum_{i=1}^\kappa \sqrt[3]{T_i}
	\end{align}
	where $T_i$, $i\in\entk$ is the area of the $i^{th}$ triangle delineated by the vertices $\pp_{i},\ve_{{i+1}},\pp_{{i+1}}$.
	\begin{figure}[hbtp]
		\centering
		\input{Figure_5.txt}
		\caption{If $K$ belongs to $\polyktau,$ the boundary of $\Dom{K}$ is tangent to every side of $K$ at $\pp_j$ for the $j^{th}$ side. The triangles hached with bricks patterns are the triangle with area $T_i$, $i\in\entk$.}\label{chap3:fig}
	\end{figure}
	Note that given $K$, determining the $\pp[\kappa]$ (and thus the $w[\kappa]$ and the $T[\kappa]$) for a random convex polygon is not a trivial question. To overcome this matter, we raise the following property on $T[\kappa]$ and $w[\kappa]$:
	
	\begin{Theoreme}\label{thm:barany1}
		Set $\kappa\geq 3$ and let $K\in \polyktau$ a polygon with side-lengths $r[\kappa]$, and internal angles $\theta[\kappa]$. The vector $(\sqrt[3]{2T_1},\ldots,\sqrt[3]{2T_\kappa})$ from equation \eqref{eq:AP1} is the unique solution $f[\kappa]$ to the following system ${\sf PS}$:
		\begin{align}\label{eq:fj}
			{\sf PS}(\theta[\kappa],r[\kappa]):\left\{
			\begin{array}{ll}\displaystyle
				f_j\left(f_j+f_{{j-1}}\right)\left(f_j+f_{\wj}\right) & = r_j\cdot r_{j+1}\cdot\sin(\theta_j),\quad \forall j\in\entk.
			\end{array}
			\right.
		\end{align}
		In this case, the $w[\kappa]$ defined in \eqref{eq:defwj} satisfies $\forall j\in\entk$, 
		\begin{align}\label{eq:wj}
			w_j=\frac{f_j}{f_j+f_{{j-1}}}.
		\end{align}
	\end{Theoreme}

	\begin{Remarque}
		If $K$ is a regular polygon with $\kappa$ sides, we have of course $r_1=\ldots=r_\kappa$ and $\theta_1=\ldots=\theta_\kappa$ and the aforementioned triangles have area $T_1=\ldots=T_\kappa=r_\kappa^2\sth/8$, and then the corresponding $f[\kappa]$ satisfies $f_1=\ldots=f_\kappa=\left(r_\kappa^2\sth/4\right)^{1/3}$.
	\end{Remarque}

	\begin{Remarque}
		Apart from special cases, the solution of ${\sf PS}$ is not expected to be given by a close formula, but by numerical methods, one can obtain a solution within any fixed precision $\varepsilon>0$. As a matter of fact, every limit shape in a figure of this paper has been computed thanks to the system ${\sf PS}$.
	\end{Remarque}
	
	In the sequel, we define the ``renormalized'' family of $f[\kappa]$ by setting for all $j\in\entk$, 
	\[g_j:=f_j/\sum_{i=1}^\kappa f_i.\]
	\Cref{thm:barany1} is actually the key to the main result of this paper:
	
	\begin{Theoreme}\label{thm:tangency}
		For any $K\in\polyktau$, we have
		\begin{align}\label{eq:res1}
			\pK \underset{n\to\infty}{\sim} C_K\cdot\frac{e^{2n}}{4^n}\frac{\AP^*(K)^{3n}}{\Area{K}^n n^{2n+\kappa/2}},
		\end{align}
		with \[C_K:=\frac{1}{(2\pi)^{\kappa/2}\sqrt{\mathbb{d}_K}}\cdot\left[\prod_{j=1}^\kappa\frac{1}{\sqrt{w_j}m_j\sin(\theta_j)r_j}\right],\]  where for all $j\in\entk$,
		\begin{align}\label{mj}
			m_j=\frac{\cotan(\theta_{{j-1}})+\cotan(\theta_j)}{r_j}(g_j+g_{{j+1}})+\frac{g_{{j+1}}+g_{j+2}}{\sin(\theta_j)r_{\wj}}+\frac{g_{{j-1}}+g_{{j}}}{\sin(\theta_{{j-1}})r_{{j-1}}},
		\end{align}
		
		and $\mathbb{d}_K$ is the determinant of the symmetric matrix $\Sigma_K^{-1}=(\sigma^{-1}_{i,j})_{1\leq i,j\leq\kappa-1}$ defined as :
		
		\begin{align}\label{eq:matrix}
			\sigma^{-1}_{j,j}&= \frac1{g_j}+\frac1{g_\kappa}+\frac{\mathbb{1}_{j\neq1}}{g_1+g_\kappa}+\frac{\mathbb{1}_{j\neq1}}{g_{j-1}+g_{j}}+\frac{\mathbb{1}_{j\neq\kappa-1}}{g_{\kappa-1}+g_\kappa}+\frac{\mathbb{1}_{j\neq\kappa-1}}{g_{j}+g_{j+1}},\quad \text{ for }j<\kappa,\\
			\sigma^{-1}_{i,j}&=\frac1{g_\kappa}+\frac{\mathbb{1}_{j=i+1}}{g_{j-1}+g_{j}}+\frac{\mathbb{1}_{i\neq 1}}{g_{\kappa}+g_1}+\frac{\mathbb{1}_{j\neq \kappa-1}}{g_{\kappa-1}+g_{\kappa}},\quad \text{ for }i<j<\kappa.
		\end{align}
	\end{Theoreme}
	In a recent paper  \cite{Morin}, the author obtained an analogous result in the case of regular convex
	$\kappa$-gons. During the elaboration of this work, it was thought that $\Ck$’s inner symmetry was the
	only path to this logarithmic equivalent. In \Cref{sec2}, a method is presented that allows us to go
	beyond this symmetry and to generalize this result for all convex polygons. However, many
	elements of this section are adapted from the regular $\kappa$-gon case, so proofs and details are
	given only when a substantial difference appears. For this reason, it is believed that this paper
	will not be understood in depth without taking a look at  \cite{Morin}.
	
	\subsection{What happens when $K$ is not in $\polytau$ ?}
	\begin{Lemma}\label{lem:rep}
		Let $\kappa\geq3$. For any $K\in\polyk$ with $\abso{\Tang{K}}=m\in\ent{3}{\kappa}$, there exists a unique ``bigger'' convex polygon $K_\mathcal{T}\supset K$ such that
		\begin{align}
			\bullet&\quad \Ktau\in \mathbf{P}^{\mathcal{T}}_m,\label{jpp8}\\
			\bullet&\quad\Dom{K}=\Dom{\Ktau},\label{jpp10}
		\end{align}
	\end{Lemma}
	
	In the sequel, with B\'ar\'any's limit shape theorem, we will identify equivalent ``classes'' of convex polygons such that for two convex polygons $K_1,K_2$ in the same class, \[{\sf Area}(K_{1})^n~\P_{K_1}(n)\underset{n\to\infty}{\sim} {\sf Area}(K_{2})^n~\P_{K_2}(n).\] As we will see, the convex sets contained in $\polytau$ can be chosen as representatives of these classes. These considerations are summarized in the following theorem:
	
	\begin{Theoreme}\label{thm:global}
		Let $K\in\mathbf{P}$. The polygon $K_\mathcal{T}\supset K$ described in \Cref{lem:rep} satisfies
		\begin{align}
			\P_{K}(n)\underset{n\to\infty}{\sim} \frac{{\sf Area}(\Ktau)^n}{{\sf Area}(K)^n}~\P_{\Ktau}(n).\label{jpp9}
		\end{align}
	\end{Theoreme}
	
	Note that ${\sf Area}(\Ktau)\geq {\sf Area}(K)$, hence the probability of being in convex position in $K$ is at least that in $\Ktau$.
	Let $m=\abso{\Tang{K}}.$ Since $K_\mathcal{T}$ is in $\mathbf{P}^{\mathcal{T}}_m$, we know an asymptotic of $\P_{K_\mathcal{T}}(n)$ as $n\to+\infty$ by \Cref{thm:tangency}. The proof of the existence of $K_\mathcal{T}$ is constructive, so that $K_\mathcal{T}$ is explicit and we may actually determine an equivalent of $\P_{K}(n)$ out of $K_\mathcal{T}$'s geometrical properties. Therefore, both Theorems \ref{thm:tangency} and \ref{thm:global} solve the question of the equivalent of $\pK$ for any convex polygon $K$. Indeed, on the one hand we have 
	\begin{align}\label{eq:res2}
		\pK &\underset{n\to\infty}{\sim} C_K\cdot\frac{e^{2n}}{4^n}\frac{\AP^*(K)^{3n}}{\Area{K}^n n^{2n+m/2}}
	\end{align}
	and on the other hand, by \eqref{jpp9}
	\begin{align}
		\P_{K}(n)\underset{n\to\infty}{\sim} C_{\Ktau}\cdot\frac{e^{2n}}{4^n}\frac{{\sf Area}(\Ktau)^n}{{\sf Area}(K)^n} \frac{\AP^*(\Ktau)^{3n}}{{\sf Area}(\Ktau)^n~n^{2n+m/2}}=C_{\Ktau}\cdot\frac{e^{2n}}{4^n}\frac{\AP^*(\Ktau)^{3n}}{{\sf Area}(K)^n~n^{2n+m/2}}.
	\end{align}
	Hence $C_K:=C_{\Ktau}$ is thus explicit and $m=\abso{\Tang{K}}$.
	The proof of \Cref{thm:global} will be provided in the third section.
	
	\subsection{A quick history.}A pretty exhaustive list of results around this question had been provided in \cite{Morin} but among those that are most directly related to our results, we ought to cite the following:
	\begin{enumerate}
		\item[$\bullet$] Valtr's formula \cite{Valtr1995} for the parallelogram $\displaystyle\mathbb{P}_\Box(n)=\frac{1}{(n!)^2}{2n-2\choose n-1}^2\underset{n\to +\infty}{\sim} \frac{1}{\pi^{2}2^5}\frac{4^{2n}e^{2n}}{n^{2n+2}},$
		\item[$\bullet$] Valtr's formula \cite{valtr1996probability} for the triangle $\displaystyle\mathbb{P}_\triangle(n)=\frac{2^n (3n-3)!}{(2n)!((n-1)!)^3}\underset{n\to +\infty}{\sim} \frac{\sqrt{3}}{4}\frac{1}{\pi^{3/2}3^3}\frac{3^{3n}e^{2n}}{2^nn^{2n+3/2}},$
		\item[$\bullet$] Marckert's formula \cite{marckert2017probability} in the disk case, and the asymptotic expansion of $\log\left(\mathbb{P}_\bigcirc(n)\right)$ by Hilhorst, Calka and Schehr \cite{hilhorst:hal-00330444},
		\item[$\bullet$] B\'ar\'any's works \cite{barany1,barany2} around the affine perimeter and the limit shape.
	\end{enumerate}
	
	Another important explicit result concerning points in convex position was found for what we call the ``bi-pointed triangle'': for $n\geq1$ points drawn uniformly in a triangle of area 1, with vertices $A,B,C$, what is the probability $Q_{\triangle}(n)$ that these points are in convex position together with $A,B$? B\'ar\'any, Rote, Steiger, Zhang \cite{Barany2000} proved that 
	\begin{align}\label{eq:bipointee}
		Q_{\triangle}(n)=\frac{2^n}{n!(n+1)!},
	\end{align} 
	a result refined by Buchta \cite{Buchta}. This formula is actually the core of
	the proof of Theorem 1.3 since, as we will see, any tuple $\zzn$ in convex position in a convex polygon may be decomposed in several bi-pointed triangles. 
	
	Notice that the study of limit shapes and convex chains took of after Vershik asked whether it was possible to determine the number and typical shape of convex lattice polygons contained in  $[-n,n]^2.$ Three different solutions were brought to light by Bárány \cite{barany3}, Vershik \cite{vershik} and Sinai \cite{sinai} in 1994. These results were refined by Bureaux, Enriquez \cite{Bureaux_2016} in 2016, and generalized in larger dimensions by Bárány, Bureaux, Lund \cite{BARANY2018143} in 2018.

	\paragraph{Notation.}In the sequel, we will reuse most of the objects we introduced in \cite{Morin} and adapt them to the general case. We introduce the following notation:
	
\begin{enumerate} 
	\item[{\it (a)}] $\Leb_{2n}$ is the Lebesgue measure on $(\RR^2)^n$.
	\item[{\it (b)}] When $K$ is a compact convex set of $\RR^2$ with non-empty interior, the set $\mathcal{D}_K(n)$ gathers all $n$-tuples of points in convex position in $K$, so that we have \[\mathbb{P}_K(n)=\Leb_{2n}(\mathcal{D}_K(n))/\Area{K}^n.\] \noindent\begin{minipage}{0.6\textwidth} \item[{\it (c)}] We choose to consider $n$-tuples $\zzn$ of points in convex {\bf canonical} order (see \Cref{chap3:fig2}); those that satisfy the following conditions\\ $\bullet$ If $(x_i,y_i)$ are the coordinates of $z_i$ in $\RR^2$, $y_1\leq y_i$ for all $i$ (that is, $z_1$ has the smallest $y$-component), and among those having the minimal $y$ component, it has the smallest $x$ component.\\ $\bullet$ The sequence $(\arg(z_{i+1}-z_i),1\leq i \leq n-1)$ is non-decreasing in $[0,2\pi]$.\\ Denote $\mathcal{C}_K(n)$ the set of such tuples of points, so that this choice induces \[\mathbb{P}_K(n)=n!~\Leb_{2n}(\mathcal{C}_K(n))/\Area{K}^n. \] \end{minipage} \hspace{1ex} \begin{minipage}{0.4\textwidth} \begin{tikzpicture}[scale=0.8]

\definecolor{darkgray176}{RGB}{176,176,176}
\definecolor{amber}{rgb}{1.0, 0.75, 0.0}

\begin{axis}[axis lines=none,
tick align=outside,
tick pos=left,
x grid style={darkgray176},
xmin=-2.5, xmax=4.5,
xtick style={color=black},
y grid style={darkgray176},
ymin=-0.4, ymax=5.5,
ytick style={color=black}
]
\addplot [->,semithick, black]
table {%
0 -0.5
0 5.1
};
\addplot [->,semithick, black]
table {%
-0.3 0
4.1 0
};
\addplot [thick, perso]
table {%
0 0
3 0
4 3
1 5
-2 4
-1 1
0 0
};
\node[] (A) at (axis cs:4.1,0.){};
\draw[above,color=black](A) node { $x$};

\node (K)at (axis cs:0,5.1){};
\draw[left,color=black](K) node { $y$};

\node[inner sep=1.5pt,circle,draw=black,fill=red] (A1) at (axis cs:0.,0.){};
\node[inner sep=1.5pt,circle,draw=black,fill=red] (A2) at (axis cs:3.,0.){};
\node[inner sep=1.5pt,circle,draw=black,fill=red] (A3) at (axis cs:4.,3.){};
\node[inner sep=1.5pt,circle,draw=black,fill=red] (A4) at (axis cs:1.,5.){};
\node[inner sep=1.5pt,circle,draw=black,fill=red] (A5) at (axis cs:-2.,4.){};
\node[inner sep=1.5pt,circle,draw=black,fill=red] (A6) at (axis cs:-1.,1.){};

\node[inner sep=1.5pt,circle,draw=black,fill=amber] (Z1) at (axis cs:1.98,0.52){};
\draw[below,color=black](Z1) node { $z_1$};
\node[inner sep=1.5pt,circle,draw=black,fill=amber] (Z2) at (axis cs:2.63,1.59){};
\draw[below,color=black](Z2) node { $z_2$};
\node[inner sep=1.5pt,circle,draw=black,fill=amber] (Z3) at (axis cs:2.69,2.73){};
\node[inner sep=1.5pt,circle,draw=black,fill=amber] (Z4) at (axis cs:2.38,3.33){};
\node[inner sep=1.5pt,circle,draw=black,fill=amber] (Z5) at (axis cs:1.13,4.51){};
\node[inner sep=1.5pt,circle,draw=black,fill=amber] (Z6) at (axis cs:-0.54,3.21){};
\node[inner sep=1.5pt,circle,draw=black,fill=amber] (Z7) at (axis cs:-0.42,3.84){};
\draw[below,color=black](Z6) node { $z_{n-1}$};
\node[inner sep=1.5pt,circle,draw=black,fill=amber] (Z8) at (axis cs:0.27,1.18){};
\draw[below,color=black](Z8) node { $z_n$};

\end{axis}

\end{tikzpicture} \captionof{figure}{The convex canonical order}\label{chap3:fig2} \end{minipage}\\ \item[{\it (d)}] $\mathsf{Q}^{(n)}_{K}$ is the distribution of an $n$-tuple $\zn$ taken under $\mathsf{U}^{(n)}_{K},$ conditioned to be in $\mathcal{C}_K(n)$.
	\item[{\it (e)}] Let $\NN_\kappa(n)=\big\{\ssk\in\NNN, \text{ such that }s_1+\ldots+s_\kappa=n \text{ and } s_{{j-1}}+s_j\neq 0 \text{ for all }j\in\entk\big\},$ the set of vectors summing to $n$ having no successives values being both zero. \end{enumerate}

	\subsection{The ``Parallel Containing Polygon'' and its coordinates}
	
	\begin{Definition}
		Fix $\kappa\geq3$, and let $K\in\polyk$. For an $n$-tuple $\zzn\in\CVnK$, we define the $\PCP(\zzn)$ as the intersection of all convex polygons with sides parallel to $K$, that contain $\zzn$. It is the minimal ``parallel containing polygon'' for the inclusion of $\zzn$. Its side-lengths are denoted by $\Cj:=\Cj(\zzn)$, with one or several $c_j$ being possibly zeroes. For all $j\in\entk$ the side-distance $\ell_j:=\ell_j(\zzn)$ denotes the smallest distance from the $j^{th}$ side of $K$ to a point in $\zzn.$\footnote{Since we will be working with random tuples $\zn$ under $\UnK$, the quantity $\Bell_j=\ell_j(\zn)$ is a.s. unique.}. See \Cref{chap3:fig12} below for a recap.
	\end{Definition}

	\begin{figure}[hbtp]
		\centering
		\begin{figure}[H]
\begin{minipage}{0.5\textwidth}
\centering
\begin{tikzpicture}[scale=0.8]

\definecolor{darkgray176}{RGB}{176,176,176}
\definecolor{amber}{rgb}{1.0, 0.75, 0.0}

\begin{axis}[axis lines=none,
tick align=outside,
tick pos=left,
x grid style={darkgray176},
xmin=-2, xmax=4.5,
xtick style={color=black},
y grid style={darkgray176},
ymin=-0.163496409913397, ymax=5.,
ytick style={color=black}
]
\addplot [thick, black]
table {%
0 0
3 0
4 3
1 5
-2 4
-1 1
0 0
};
\addplot [thick, blue]
table {%
0.92 0.52
2.27 0.52
3.15 3.16
1.13 4.51
-0.77 3.88
0.07 1.38
0.92 0.52
};
\addplot [<->,thick, red]
table {%
3.29 0.86
2.47 1.13
};
\addplot [<->,thick, red]
table {%
1.43 0
1.43 0.52
};
\addplot [<->,thick, red]
table {%
-0.46 0.46
0.27 1.18
};
\node[] (A) at (axis cs:1.43,0.25){};
\draw[right,color=red](A) node { $\ell_1$};

\node (K)at (axis cs:-0.05,0.9){};
\draw[left,color=red](K) node { $\ell_\kappa$};	

\node[] (T) at (axis cs:2.95,0.95){};
\draw[above,color=red](T) node { $\ell_2$};

\node[] (C2) at (axis cs:2.86,2.29){};
\draw[right,color=blue](C2) node { $c_2$};

\node (C1)at (axis cs:1.17,0.52){};
\draw[above right,color=blue](C1) node { $c_1$};	

\node[] (CK) at (axis cs:0.7,0.79){};
\draw[above,color=blue](CK) node { $c_\kappa$};

\node[inner sep=1.5pt,circle,draw=black,fill=amber] (A1) at (axis cs:1.98,0.52){};
\node[inner sep=1.5pt,circle,draw=black,fill=amber] (A2) at (axis cs:2.63,1.59){};
\node[inner sep=1.5pt,circle,draw=black,fill=amber] (A3) at (axis cs:2.69,2.73){};
\node[inner sep=1.5pt,circle,draw=black,fill=amber] (A4) at (axis cs:2.38,3.33){};
\node[inner sep=1.5pt,circle,draw=black,fill=amber] (A5) at (axis cs:1.13,4.51){};
\node[inner sep=1.5pt,circle,draw=black,fill=amber] (A6) at (axis cs:-0.54,3.21){};
\node[inner sep=1.5pt,circle,draw=black,fill=amber] (A7) at (axis cs:-0.42,3.84){};
\node[inner sep=1.5pt,circle,draw=black,fill=amber] (A8) at (axis cs:0.27,1.18){};

\end{axis}

\end{tikzpicture}
\end{minipage}
\quad
\begin{minipage}{0.5\textwidth}
\centering
\begin{tikzpicture}[scale=0.8]

\definecolor{darkgray176}{RGB}{176,176,176}
\definecolor{amber}{rgb}{1.0, 0.75, 0.0}

\begin{axis}[axis lines=none,
tick align=outside,
tick pos=left,
x grid style={darkgray176},
xmin=-2, xmax=4.5,
xtick style={color=black},
y grid style={darkgray176},
ymin=-0.163496409913397, ymax=5.,
ytick style={color=black}
]
\addplot [thick, black]
table {%
0 0
3 0
4 3
1 5
-2 4
-1 1
0 0
};
\addplot [thick, blue]
table {%
1.98 0.52
2.27 0.52
3.15 3.16
1.13 4.51
0.69 4.37
1.98 0.52
};

\node[inner sep=1.5pt,circle,draw=black,fill=amber] (A1) at (axis cs:1.98,0.52){};
\node[inner sep=1.5pt,circle,draw=black,fill=amber] (A2) at (axis cs:2.63,1.59){};
\node[inner sep=1.5pt,circle,draw=black,fill=amber] (A3) at (axis cs:2.69,2.73){};
\node[inner sep=1.5pt,circle,draw=black,fill=amber] (A4) at (axis cs:2.38,3.33){};
\node[inner sep=1.5pt,circle,draw=black,fill=amber] (A5) at (axis cs:1.13,4.51){};

\end{axis}

\end{tikzpicture}
\end{minipage}

\end{figure}
		\caption{
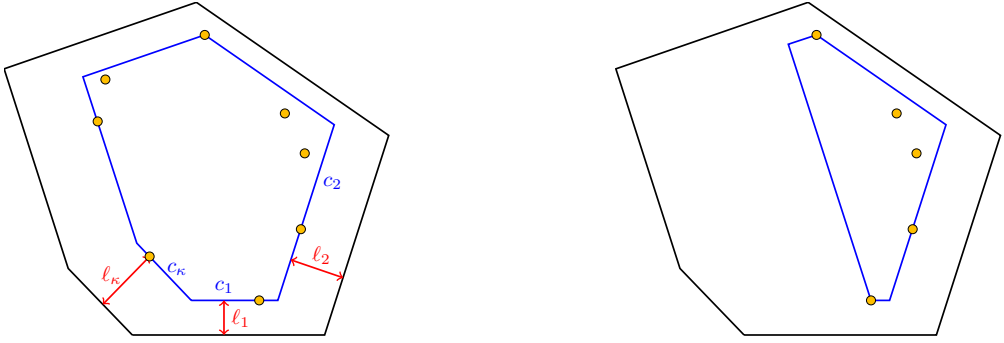
			Two example of $\PCP$ for some $K\in\mathbf{P}_6$ (drawn in black). The $\PCP(\zzn)$ is drawn in blue, its side-lengths $\Cj$ are represented in blue as well, and the side-distances $\ell[\kappa]$ are drawn in red (on the left only). On the left picture, all side-lengths $c[\kappa]$ are nonzero whereas it is not the case anymore on the right picture.}
		\label{chap3:fig12}
	\end{figure}

	\paragraph{``Contact points'' (see \Cref{fig:super}).}
	For each $\zzn$ in $\CVnK$, each side of $\PCP(\zzn)$ contains at least one element of $\{z_1,\cdots,z_n\}$. The $j^{th}$ ``contact point'' $\cp_j(\zzn)$ is the\footnote{Once again, the point on each side is a.s. unique} point of $\{z_1,\cdots,z_n\}$, which is on the $j^{th}$ side of $\PCP(\zzn)$, and which is the smallest for the lexicographical order among those with this property. 
	Note that $\cp_j=\cp_{\wj}$ is possible, and has a positive probability for all $n\geq1$ (this is the case $\cp_3=\cp_4$ in \Cref{fig:super}).
	
	Denote by $\bb_j$ the intersection point between the $j^{th}$ and $\wj^{th}$ sides of $\PCP(\zzn)$ for all $j\in\entk$ (the $j^{th}$ vertex of the $\PCP(\zzn)$). In the case where the $j^{th}$ side of $\PCP(\zzn)$ is reduced to a point, \ie $c_j=0,$ we have $\cp_{{j-1}}=\bb_{{j-1}}=\cp_j=\bb_j=\cp_{\wj}$. 
	
	The triangle $\cp_j,\cp_{\wj},\bb_j$ will be refered to as the $j^{th}$ corner of $\PCP(\zzn)$ or $\corner_j(\zzn)$.
	For all $j\in \entk$, let $k:=k(j)\in\entn$ such that $z_k=\cp_j$ and denote by $s_j:=s_j(z[n])$ the integer such that $z_{k+s_j}=\cp_{\wj}$ (eventually $s_j=0$); the quantity $s_j$ denotes the number of vectors joining the points of the convex chain $(z_k=\cp_j,\ldots,z_{k+s_j}=\cp_{\wj})$. We will refer to the tuple $\ssk$ as the {\bf size-vector}.
	
	\begin{Remarque}\label{rem:equ}
		A quick glance at \Cref{fig:super} and the previous considerations allows one to see that $s_j=0$ is equivalent to $\cp_j=\cp_{\wj}$. Therefore, $c_j=0$ is equivalent to $s_{{j-1}}+s_{j}=0.$
	\end{Remarque}
	{\begin{figure}[H]
	\centering
	\begin{tikzpicture}[scale=0.8]
		
		\definecolor{darkgray176}{RGB}{176,176,176}
		\definecolor{amber}{RGB}{34,139,34}
		
		\begin{axis}[axis lines=none,
			tick align=outside,
			tick pos=left,
			x grid style={darkgray176},
			xmin=-2, xmax=4.5,
			xtick style={color=black},
			y grid style={darkgray176},
			ymin=-0.163496409913397, ymax=5.,
			ytick style={color=black}
			]
			\addplot [thick, black]
			table {%
				0 0
				3 0
				4 3
				1 5
				-2 4
				-1 1
				0 0
			};
			\addplot [thick, blue]
			table {%
				0.92 0.52
				2.27 0.52
				3.15 3.16
				1.13 4.51
				-0.77 3.88
				0.07 1.38
				0.92 0.52
			};

			\node[inner sep=1.pt,circle,draw=blue,fill=blue] (B1) at (axis cs:0.92,0.52){};
			\draw[left,color=blue](B1) node {$\bb_\kappa$};	
			\node[inner sep=1.pt,circle,draw=blue,fill=blue] (B2) at (axis cs:2.27,0.52){};
			\draw[right,color=blue](B2) node {$\bb_1$};	
			\node[inner sep=1.pt,circle,draw=blue,fill=blue] (B3) at (axis cs:3.15,3.16){};
			\node[inner sep=1.pt,circle,draw=blue,fill=blue] (B4) at (axis cs:1.13,4.51){};
			\node[inner sep=1.pt,circle,draw=blue,fill=blue] (B5) at (axis cs:-0.77,3.88){};
			\node[inner sep=1.pt,circle,draw=blue,fill=blue] (B6) at (axis cs:0.07,1.38){};

			\node[inner sep=1.5pt,circle,draw=black,fill=amber] (A1) at (axis cs:1.98,0.52){};
			\draw[below,color=amber](A1) node {$\cp_1$};	
			\node[inner sep=1.5pt,circle,draw=black,fill=amber] (A2) at (axis cs:2.63,1.59){};
			\draw[right,color=amber](A2) node {$\cp_2$};	
			\node[inner sep=1.5pt,circle,draw=black,fill=amber] (A5) at (axis cs:1.13,4.51){};
			\node[inner sep=1.5pt,circle,draw=black,fill=amber] (A7) at (axis cs:-0.54,3.21){};
			\node[inner sep=1.5pt,circle,draw=black,fill=amber] (A8) at (axis cs:0.27,1.18){};
			
			\fill[color=blue,pattern=north east lines] (A2.center)--(B3.center)--(A5.center)--cycle;
			\node[inner sep=1.5pt,circle,draw=black,fill=amber] (A3) at (axis cs:2.69,2.73){};
			\node[inner sep=1.5pt,circle,draw=black,fill=amber] (A4) at (axis cs:2.38,3.33){};
			
			\fill[color=blue,pattern=north east lines] (A5.center)--(B5.center)--(A7.center)--cycle;
			\node[inner sep=1.5pt,circle,draw=black,fill=amber] (A6) at (axis cs:-0.42,3.84){};
			
			\draw[amber,line width=1pt] (A1)--(A2)--(A3)--(A4)--(A5)--(A6)--(A7)--(A8)--(A1);
			\draw[blue,line width=0.5pt] (A2)--(A5)--(A7);

		\end{axis}
		
	\end{tikzpicture}
	\caption{In $K$, an example of $\zzn$-gon, the $\PCP(\zzn)$ and its vertices $\bb[6]$, and the second and third corners (the hashed areas). Here, the size-vector is $s[6]=(1,3,0,2,1,1)$.}\label{fig:super}
\end{figure}}
	
	The $\PCP$ is sort of an equivalent to the $\ECP$ that we had in the regular $\kappa$-gon case.
	The geometric equations of the $\PCP$, that may be obtained essentially with trigonometric considerations and several applications of Thales' theorem, are given by the following proposition:
	\begin{Proposition}\label{prop2}
		Let $z[n]\in\mathcal{C}_K(n)$, and $\Cj,\Lj$ be the associated quantities with $\PCP(\zzn)$.
		\begin{enumerate}
			\item[{\it (i)}] The $\Lj$ and $\Cj$ are related by the $\kappa$ equations
			\begin{equation}\label{eq:EEj2} c_j=r_j-{\cl}_j(\Lj),\quad \forall j \in \{1,\ldots,\kappa\},\end{equation} where
			\[{\cl}_j(\Lj)=\frac{\ell_{{j-1}}}{\sin(\theta_{{j-1}})}+\frac{\ell_{{j+1}}}{\sin(\theta_j)}+\ell_j\left(\cotan(\theta_{{j-1}})+\cotan(\theta_j)\right).\]
			\item[{\it (ii)}] The set $\mathcal{L}_K=\ell[\kappa]\left(\mathcal{C}_K(n)\right)$ (of all possible vectors $\ell[\kappa]$) is the set of solutions $\Lj$ to the inequations
				\begin{equation}\label{eq:EEj3}
				r_j-{\cl}_j(\Lj)\geq0\quad \forall j \in \{1,\ldots,\kappa\}.\end{equation}
			together with the conditions $\ell_j\geq0,j\in\entk.$
		\end{enumerate}
	\end{Proposition}

	\begin{proof}
		The proof relies on basic trigonometric considerations. One can find a similar proof (to be adapted) in \cite{Morin}{Proposition 2.2}.
	\end{proof}

	\section{A result on tangent limit shapes and polygons}\label{sec2}
	
	Consider the Hausdorff distance $d_H$, and for any convex set $K$ and any $\varepsilon>0,$ set 
	\begin{align}\label{not:sng}
		S_n(K,\varepsilon)=\left\{z[n]\in{\mathcal{C}_K}(n)\text{ s.t. } d_H(\CH(\zzn),\Dom{K})\leq\varepsilon\right\}.
	\end{align}
	We recall B\'ar\'any's limit shape theorem, that will be of paramount importance for the sequel of this paper:
	\begin{Theoreme}[Limit shape theorem \cite{barany1}]\label{thm:lst}
		For any convex set $K$ and any $\varepsilon>0,$
		\begin{align}
			\frac{\Leb_{2n}\left(S_n(K,\varepsilon)^c\right)}{\Leb_{2n}\left({\mathcal{C}_K}(n)\right)}\underset{n\to\infty}{\longrightarrow}0,
		\end{align}
		from which we deduce that $\displaystyle \P_{K}(n)\underset{n\to\infty}{\sim} \Leb_{2n}\left(S_n(K,\varepsilon)\right)/\Area{K}^n$.
	\end{Theoreme}
	
	In words, this theorem states that, when $n$ is large, the overwhelming majority of $n$-tuples that are in convex position in a convex set $K$ gather around the boundary of $\Dom{K}.$
	
	To evaluate the asymptotic equivalent of $\pK$ for $K\in\polytau$, we need to understand the behaviour of $(\lk,\sk):=(\Lj(\zn),\ssk(\zn))$ (that is, the coordinates of $\PCP(\zn)$) for $\zn$ that is $\mathsf{Q}^{(n)}_{K}$-distributed.

	\subsection{The joint distribution of $(\lk,\sk)$}
	
	As anounced in the introduction, in this section we will consider polygons $K\in\polytau$, that is, such that the set $\Dom{K}$ is {\bf tangent to all sides} of $K$.
	\paragraph{Notation.}For all polygons $K\in\polytau$, we define $\ptK$, the probability that an $n$-tuple $\zn$ taken under $\UnK$ is in convex position and has a ``full-sided'' $\PCP$ (\ie if $K\in\polyktau$, $c_j(\zn)>0$ for all $j\in\entk$).
	\begin{Lemma}\label{lem:dvar}
		For all $K\in\polytau,$
		\[\pK\underset{n\to+\infty}{\sim}\ptK.\]
	\end{Lemma}
	
	\begin{proof} The proof is similar to that of the $K=\Ck$ case.
		For an $n$-tuple $\zn$ under $\UnK$, B\'ar\'any's \Cref{thm:lst} implies immediately that
		\[\mathbb{P}\big(\PCP(\zn)\text{ is ``full-sided''}~\vert~ \zn\in\mathcal{C}_K(n)\big)=\frac{\ptK}{\pK}\underset{n\to+\infty}{\longrightarrow}1.\]
		
		We have $\Dom{K}\subset K$. Hence we can define the $\PCP$ of $\Dom{K}$ as the minimum parallel polygon to $K$ that contains $\Dom{K}$. Since $K\in\polytau$, $\Dom{K}$ is tangent at all sides of $K$, which means that $\PCP(\Dom{K})=K$. Since the $\PCP$ is a continuous map in the plane, we have
		\[d_H(\zn,\Dom{K})\cvg0 \implies d_H(\PCP(\zn),K)=d_H(\PCP(\zn),\PCP(\Dom{K}))\cvg0.\]
		Hence, when $n$ is large, the $\PCP$ has a.s. as many sides as $K$.
	\end{proof}
	
	This lemma shows that the data of $\zn$ with ``full-sided'' $\PCP$s are sufficient for our search of $\pK$. Therefore, a good characterization of this quantity is required:
	\begin{Lemma}
		Let $K\in\polyktau$, and let $\zn$ under $\mathsf{Q}^{(n)}_{K}$. The two following assumptions are equivalent:
		\begin{enumerate}
			\item[{\it (i)}] $\PCP(\zn)$ is ``full-sided'',
			\item[{\it (ii)}] $\sk=\ssk(\zn)\in\Nkn$ (this is a condition on the number of vectors in each corner).
		\end{enumerate}
	\end{Lemma}
	\begin{proof}
		This is pretty much \Cref{rem:equ}: suppose that the $\PCP(\zzn)$ has exactly $\kappa$ nonzero sides, \ie if $\mathbf{c}[\kappa]=c[\kappa](\zn)$, we have $\mathbf{c}_j>0$ for all $j\in\entk.$ Inside the tuple $\zn$, consider for all $j\in\entk$ the contact points $\Bcp_{{j-1}}$, $\Bcp_{j}$ and $\Bcp_{\wj}.$ A small picture suffices to see that we cannot have $\Bcp_{{j-1}}=\Bcp_{j}=\Bcp_{\wj}$ for this is equivalent to $\mathbf{c}_j=0,$ and thus is also equivalent to the fact that there exists a nonzero vector leading either $\Bcp_{{j-1}}$ to $\Bcp_{j}$ (\ie $\mathbf{s}_{{j-1}}\geq1$), or $\Bcp_{j}$ to $\Bcp_{\wj}$ (\ie $\mathbf{s}_j\geq1$).
	\end{proof}
	\paragraph{Notation.}In accordance with this lemma, we define the distribution $\QtnK$ that denotes the law of an $n$-tuple $\zn$ taken under $\mathsf{Q}^{(n)}_{K}$, conditioned to have $\ssk(\zn)\in\Nkn$, that is, conditioned to have a full-sided $\PCP$. In particular, with $d_V$ the total variation distance, by \Cref{lem:dvar} we have 
	\[d_V\left(\mathsf{Q}^{(n)}_{K},\QtnK\right)\cvg 0.\]
	For such a $\zn$ taken under $\QtnK,$ we give in the following theorem the joint distribution of $(\lk=\Lj(\zn),\sk=\ssk(\zn))$, which is analogous to \cite{Morin}[Theorem 3.6.].
	
	\begin{Theoreme}\label{thm:distri2}
		Let $K\in\polyktau$ and let $\zn$ be an $n$-tuple of random points under $\QtnK$, and consider the random variables $\lk=\Lj(\zn)$, $\sk=\ssk(\zn).$ Then for a given $\ssk\in\Nkn$, the pair $(\lk,\sk)$ has the joint distribution 
		\begin{align}\label{equmonstre2}
			\mathbb{P}\left(\lk\in\mathrm{d}\Lj,\sk=\ssk\right)=f^{(n)}_{K}\left(\ell[\kappa],\ssk\right)\mathrm{d}\ell[\kappa],
		\end{align}where we denote $\mathrm{d}\ell[\kappa]=\prod_{j=1}^\kappa\mathrm{d}\ell_j$, and where
		\begin{align}\label{jpp}
			f^{(n)}_{K}\left(\ell[\kappa],\ssk\right)=\frac{n!}{\ptK} \mathbb{1}_{\Lj\in\mathcal{L}_K}	\prod_{j=1}^{\kappa}\frac{\sin(\theta_j)^{s_j-1}c_j^{s_{{j-1}}+s_j-1}}{s_j!(s_{{j-1}}+s_j-1)!}.
		\end{align}
		Recall that, by \eqref{eq:EEj2}, $\Cj$ is indeed a function of $\Lj.$
	\end{Theoreme}
	
	\begin{proof}
		We provide a proof of this result which is equivalent to but different from that of \cite{Morin}[Theorem 3.6.].
		
		Recall $\mathcal{C}_K(n)$ defined in page \pageref{chap3:fig2}, and let us rewrite $f^{(n)}_{K}$ as
		\[f^{(n)}_{K}\left(\ell[\kappa],\ssk\right)\mathrm{d}\ell[\kappa]=\frac{\Leb_{2n}\left(\left\{\zzn\in\mathcal{C}_K(n)\text{ s.t. }\ssk(\zzn)=\ssk,\Lj(\zzn)\in\mathrm{d}\ell[\kappa]\right\}\right)}{\ptK}.\]
		Let us compute the numerator of this expression. Conditional on a $\PCP$ with side-lengths $\Cj$ and its corners, $\zzn$ is in convex position if and only if, for all $j\in\entk$, the $s_j-1$ points in the $j^{th}$ corner form a convex chain, together with the contact points (we prove it in \cite{Morin}[Lemma 2.3.]) .
		
		Assume now for every $j\in\entk$, $\cp_j=\bb_j+x_j(\bb_{\wj}-\bb_j)$, that is, positioned at the coordinate $x_jc_j$ on the $j^{th}$ side of the $\PCP$, with $x_j\in\zerun$. Then the area of $\corner_j$ (seen as a triangle) is $(1-x_j)c_j x_{\wj}c_{\wj}\sin(\theta_j)/2$. 
		
		Hence, recalling \eqref{eq:bipointee} for the bi-pointed triangle, the Lebesgue measure of convex chains in $\corner_j$ is  \[\frac{2^{s_j-1}}{(s_j-1)!s_j!}\left(\frac{(1-x_j)c_j x_{\wj}c_{\wj}\sin(\theta_j)}{2}\right)^{s_j-1}.\]
		There are ${n \choose s_1,\ldots,s_\kappa}$ ways of choosing the vertices in the corners so to respect the size-vector $\ssk$. We now have to integrate on the position of the contact points on their respective sides:
		\begin{align*}
			\ptK\cdot f^{(n)}_{K}\left(\ell[\kappa],\ssk\right)&={n \choose s_1,\ldots,s_\kappa}\int_{0}^{1}\cdots\int_{0}^{1}\left[\prod_{j=1}^{\kappa} \left({(1-x_j)c_j x_{\wj}c_{\wj}\sin(\theta_j)}\right)^{s_j-1}\frac{c_j}{(s_j-1)!s_j!}\mathrm{d}x_j\right].
		\end{align*}
		With the classical integral identity (relative to beta distributions) $\displaystyle \int_{0}^{1}x^a(1-x)^b\mathrm{d}x=\frac{a!b!}{(a+b+1)!}$, it comes
		\begin{align*}
			\ptK\cdot f^{(n)}_{K}\left(\ell[\kappa],\ssk\right)&= n! \prod_{j=1}^{\kappa} \frac{\sin(\theta_j)^{s_j-1}c_j^{s_j+s_{\wj}-1}}{(s_j+s_{\wj}-1)!s_j!}.
		\end{align*}
		
	\end{proof}
	
	\subsection{The quantity $\AP^*(K)$ in $f^{(n)}_K$}
	
	B\'ar\'any's logarithmic equivalent of $\pk$ is expressed in terms of the quantity $\AP^*(K)$. The next step of our strategy consists in understanding how to retrieve this quantity $\AP^*(K)$ in $f^{(n)}_K$.\par
	
	We want to describe the asymptotic behaviour of the measure computed in \Cref{thm:distri2}. To do so, we will make a change of variables in order to make apparent the asymptotics of $f^{(n)}_K$ (that is, we want to express $f^{(n)}_{K}$ in terms of simple distributions), so that, eventually, its convergence will appear as a consequence of standard considerations. 
	Let us rewrite 
	\begin{align}\label{jpp2}
		f^{(n)}_{K}\left(\ell[\kappa],\ssk\right)=\frac{n!}{\ptK\prod_{j=1}^\kappa \sin(\theta_j)r_j}\cdot\Pi\cdot\mathbb{1}_{\Lj\in\mathcal{L}_K},
	\end{align}
	where $\displaystyle\Pi := \prod_{j=1}^{\kappa}\frac{r_j\cdot\sin(\theta_j)^{s_j}c_j^{s_{{j-1}}+s_j-1}}{s_j!(s_{{j-1}}+s_j-1)!}$. By \eqref{eq:EEj2}, the quantity $\Pi$ may be rewritten as
	\begin{align}\label{eq:jpp7}
		\Pi=
		\left[\prod_{j=1}^{\kappa}\frac{\left(r_jr_{\wj}\sin(\theta_j)\right)^{s_j}}{s_j!(s_{{j-1}}+s_j-1)!}\right]\cdot\prod_{j=1}^{\kappa} \left(1-\frac{\cl_j(\Lj)}{r_j}\right)^{s_{{j-1}}+s_j-1}.
	\end{align}
	
	We will interprete the first product in the rhs of \eqref{eq:jpp7} in terms of Poisson r.v. coinciding on $\ssk$:
	for two families of positive real numbers $f^{(1)}[\kappa],f^{(2)}[\kappa]$, set $\bar{f}^{(1)}=\sum_{i=1}^\kappa f^{(1)}_i$, and let $\kappa$ Poisson r.v. $\mathcal{P}^{(n)}_{1,1},\ldots,\mathcal{P}^{(n)}_{1,\kappa}$ of respective parameters $n f^{(1)}_1/\bar{f}^{(1)},\ldots,n f^{(1)}_\kappa/\bar{f}^{(1)}$, and $\kappa$ other Poisson r.v. $\mathcal{P}^{(n)}_{2,1},\ldots,\mathcal{P}^{(n)}_{2,\kappa}$ of respective parameters $nf^{(2)}_1/\bar{f}^{(1)},\ldots,nf^{(2)}_\kappa/\bar{f}^{(1)}$, all the r.v. being independent.
	
	For $\ssk\in\Nkn,$ the probability $\mathbb{M}(s[\kappa])=\P\left(\mathcal{P}^{(n)}_{1,j}=s_j,\mathcal{P}^{(n)}_{2,j}=s_{{j-1}}+s_j-1,j\in\entk\right)$ satisfies
	\begin{align}\label{jpp3}\nonumber
		\mathbb{M}(s[\kappa])&=\prod_{j=1}^{\kappa}\frac{e^{-nf^{(1)}_j/\bar{f}^{(1)}}(nf^{(1)}_j/\bar{f}^{(1)})^{s_j}}{s_j!}
		\cdot\frac{e^{-nf^{(2)}_j/\bar{f}^{(1)}}(nf^{(2)}_j/\bar{f}^{(1)})^{s_{{j-1}}+s_j-1}}{(s_{{j-1}}+s_j-1)!}\\
		&=\frac{n^{3n-\kappa} e^{-\lambda}}{(\bar{f}^{(1)})^{3n-\kappa}\prod_{j=1}^\kappa f^{(2)}_j}\prod_{j=1}^{\kappa}\frac{\left(f^{(1)}_jf^{(2)}_{\wj}f^{(2)}_j\right)^{s_j}}{s_j!(s_{{j-1}}+s_j-1)!}
	\end{align}
	where $\displaystyle\lambda =\frac{n}{\bar{f}^{(1)}}\sum_{j=1}^{\kappa} \left(f^{(1)}_j+f^{(2)}_j\right)$, so that the last term of the equation is very similar to \eqref{eq:jpp7} if we choose carefully $(f^{(1)}[\kappa],f^{(2)}[\kappa])$. In order to fit the numerator in \eqref{eq:jpp7} and \eqref{jpp3}, it suffices to pick $f^{(1)},f^{(2)}$ so that they solve the following system:
	
	\begin{align}\label{jpp14}
		\left\{
		\begin{array}{rl}\displaystyle
			f^{(2)}_j & = f^{(1)}_j+f^{(1)}_{{j-1}} \\
			f^{(1)}_jf^{(2)}_{\wj}f^{(2)}_j& = r_j\cdot r_{j+1}\cdot\sin(\theta_j),\quad j\in\entk.
		\end{array}
		\right.
	\end{align}
	The variables $f^{(2)}$ can be eliminated from \eqref{jpp14} so that, we retrieve the parametrizing system ${\sf PS}(\theta[\kappa],r[\kappa])$ stated in \eqref{eq:fj}:
	\begin{align*}
		{\sf PS}(\theta[\kappa],r[\kappa]):\left\{
		\begin{array}{ll}\displaystyle
			f^{(1)}_j\left(f^{(1)}_j+f^{(1)}_{{j-1}}\right)\left(f^{(1)}_j+f^{(1)}_{{j+1}}\right) & = r_j\cdot r_{j+1}\cdot\sin(\theta_j),\quad j\in\entk.
		\end{array}
		\right.
	\end{align*}
	
	\paragraph{The link with the affine perimeter.}
	Let us assume for a moment that we found a positive vector $(f[\kappa]:=f_1[\kappa])$ solution to ${\sf PS}$.
	
	We may now prove \Cref{thm:barany1}, which states that for all $K\in \polyktau$,  $f_j=\sqrt[3]{2T_j},$ for all $j\in\entk$. To begin with, we start by proving that the supremum over all affine perimeters is given by
	\begin{align}\label{eq:APAP}
		\AP^*(K)=2^{2/3}\sum_{i=1}^\kappa f_i.
	\end{align}
	
	\begin{proof}[Proof of \Cref{thm:barany1}]
		For any convex set $K$, recall \Cref{thm2} where the quantity $\AP^*(K)$ is defined as $\displaystyle\AP^*(K):=\max_{\substack{S \text{convex set} \\ S\subset K}} \AP(S).$ 
		Let $\ST$ be the set of convex sets contained in $K$, such that their boundary is tangent at all sides of $K$, and are composed of parabolic arcs between the tangency points. Since $K$ is in $\polyktau$, by B\'ar\'any's limit shape theorem, we know that $\Dom{K}\in \ST.$	
		Therefore, $\AP^*(K)$ may be rewritten as
		\begin{align}\label{eq:AP4}
			\AP^*(K)&=\max_{\mathcal{C}\in\ST}\AP(\mathcal{C}).
		\end{align}
		Let $\mathcal{C}\in \ST$ be one of these sets. We denote by $\pp_j=\pp_j(\mathcal{C})$ the point of tangency of $\mathcal{C}$ with the $j^{th}$ side of $K$. The affine perimeter of $\mathcal{C}$ is  given, by definition, as 
		\begin{align}\label{eq:AP}
			\AP(\mathcal{C})=2\sum_{i=1}^\kappa \sqrt[3]{T_i},
		\end{align}
		where $T_i=T_i(\mathcal{C}),i\in\entk$ denotes the area of the triangle delineated by the vertices $\pp_{i},\ve_{{i+1}},\pp_{{i+1}}$ (see \Cref{chap3:fig17}). 
		{\begin{figure}[h]
				\centering
				\input{Figure_17.txt}
				\caption{If $K$ is in $\polyktau,$ pick $\mathcal{C}\in\ST$ a curve (drawn in green) that is tangent to every side of $K$ at $\pp_j$ for the $j^{th}$ side. This curve is not necessarily the boundary of $\Dom{K}.$ The triangles hached with dots are the triangle $T_i$, $i\in\entk$.} \label{chap3:fig17}
			\end{figure}
		}
		
		Set $u_j=u_j(\mathcal{C}):=\norm{\ve_j-\pp_j}/r_j$. With this notation, rewriting the area $T_i$ in terms of the geometric properties of $K$ and $\mathcal{C}$ turns \eqref{eq:AP} into
		\begin{align}\label{eq:AP2}\nonumber
			\AP(\mathcal{C})&=2\sum_{i=1}^\kappa \left(r_iu_i\cdot r_{{i+1}}(1-u_{{i+1}})\cdot \sin(\theta_i)/2\right)^{1/3}\\
			&=2^{2/3}\psi(r[\kappa],\theta[\kappa],u[\kappa]),
		\end{align}
		where $\psi(r[\kappa],\theta[\kappa],u[\kappa]):=\sum_{i=1}^\kappa \left(r_iu_i\cdot r_{{i+1}}(1-u_{{i+1}})\cdot \sin(\theta_i)\right)^{1/3}$. Therefore, the quantity $\AP^*(K)$ is, by \eqref{eq:AP4}, given by
		\begin{align}\label{eq:AP3}
			\AP^*(K)&=2^{2/3}\max_{{u[\kappa]\in\zerun^\kappa}}\psi(r[\kappa],\theta[\kappa],u[\kappa]),
		\end{align}
		since taking the supremum on all $u[\kappa]\left(:=u[\kappa](\mathcal{C})\right)$ results in taking the supremum on the $u[\kappa]$ in $\zerun$ (and it becomes a $\max$ since we know $\AP^*(K)$ is attained for some $\mathcal{C}$).
		By \eqref{eq:fj}, we have \[\psi(r[\kappa],\theta[\kappa],u[\kappa])=\sum_{i=1}^{\kappa}\left[f_i\left(f_i+f_{{i-1}}\right)u_i\left(f_i+f_{{i+1}}\right)(1-u_{{i+1}})\right]^{1/3}.\]

		By the classical convexity inequality $\displaystyle (abc)^{1/3}\leq (a+b+c)/3$ for all $a,b,c\geq0$, we easily see that for all ${u[\kappa]\in\zerun^\kappa}$, we have 
		\[\psi(r[\kappa],\theta[\kappa],u[\kappa])\leq\frac1{3}\sum_{i=1}^{\kappa}\left[f_i+\left(f_i+f_{{i-1}}\right)u_i+\left(f_i+f_{{i+1}}\right)(1-u_{{i+1}})\right]=\sum_{i=1}^{ \kappa}f_i. \]
		Hence, \eqref{eq:AP3} yields $\AP^*(K)\leq2^{2/3}\sum_{i=1}^{ \kappa}f_i.$
		Set now $\displaystyle w_j:=\frac{f_j}{f_{{j-1}}+f_j}$ for all $j\in\entk$. We have $\psi(r[\kappa],\theta[\kappa],w[\kappa])=\sum_{i=1}^{\kappa}f_i$, which yields \eqref{eq:APAP} and \eqref{eq:wj}.
		
		\par Now, substituting $w[\kappa]$ in ${\sf PS}$ gives 
		\[f_j^3=w_jr_j(1-w_{\wj})r_{\wj}\sin(\theta_j)=2T_j,\text{ for all }j\in\entk.\] 
		Hence, if $f[\kappa]$ is a solution to ${\sf PS}(\theta[\kappa],r[\kappa])$, it is equal to $((2T_1)^{1/3},\ldots,(2T_\kappa)^{1/3})$, where the quantities $T[\kappa]$ are those of the maximizing curve $\mathcal{C}=\Dom{K}$ which is unique. It means that a solution to ${\sf PS}(\theta[\kappa],r[\kappa])$ is unique. Conversely, we can check that $(\sqrt[3]{2T_1},\ldots,\sqrt[3]{2T_\kappa})$ is solution to ${\sf PS}$ with such $w[\kappa]$. This gives \Cref{thm:barany1}.
		
	\end{proof}
	
	\subsection*{Proof of \Cref{thm:tangency}:}
	
	Set $\bar{f}=\sum_{i=1}^{\kappa}f_i$, and the family $g[\kappa]$ of real numbers defined for all $j\in\entk$ as \[g_j:=f_j/\bar{f}.\] For the $\lambda$ defined below \eqref{jpp3}, we have $\lambda=3n$, and we rewrite the joint ``distribution'' $f^{(n)}_K$ as 
	\begin{align}\label{eq:rewritingjoint}
		f^{(n)}_{K}\left(\ell[\kappa],\ssk\right)=\Faktor(n)\mathbb{1}_{\Lj\in\mathcal{L}_K}\mathbb{M}(s[\kappa])\cdot\prod_{j=1}^\kappa\left(1-\frac{\cl_j(\Lj)}{r_j}\right)^{s_{{j-1}}+s_j-1}.
	\end{align}
	where $\displaystyle \Faktor(n):=\left(\prod_{j=1}^\kappa\frac{f_j+f_{{j-1}}}{\sin(\theta_j)r_j}\right)\cdot\frac{n!~e^{3n}(\bar{f})^{3n-\kappa}}{\ptK~n^{3n-\kappa}}$, and $\mathbb{M}(s[\kappa])$ was introduced in \eqref{jpp3}.
	
	We get to the main theorem of this section, which describes the asymptotic behaviour of the renormalized geometric variables of the $\PCP.$
	\begin{Theoreme}\label{thm:commeavant}
		Let $\zn$ under $\QtnK$, and consider $\lk=\Lj(\zn),\sk=\ssk(\zn)$. 
		We introduce the random variables $\widebar{\Bell}^{(n)}[\kappa]=n\lk$ and $\displaystyle\mathbf{x}_j^{(n)}=({\snj-ng_j})/{\sqrt{n}}$, for all $j\in\ent{1}{\kappa-1}$.
		The following convergence in distribution holds in $\RR^{2\kappa-1}$:
		\[\left(\widebar{\Bell}_1^{(n)},\ldots,\widebar{\Bell}_\kappa^{(n)},\mathbf{x}_1^{(n)},\ldots,\mathbf{x}_{\kappa-1}^{(n)}\right)\dd\left(\widebar{\Bell}_1,\ldots,\widebar{\Bell}_\kappa,\mathbf{x}_1,\ldots,\mathbf{x}_{\kappa-1}\right),\]
		where the variables $\widebar{\Bell}[\kappa]$ are independent from the $\mathbf{x}[\kappa-1]$, the r.v. $\widebar{\Bell}_j$ is exponentially distributed with rate $m_j$ (see equation \eqref{mj}), and $\mathbf{x}=\mathbf{x}[\kappa-1]$ is a centered Gaussian random vector whose inverse covariance matrix $\Sigma_K^{-1}$ and $\mathbb{d}_K=\det\left(\Sigma_K^{-1}\right)$ were given in \eqref{eq:matrix}. Note that the joint density $\chi_K$ of $\left(\widebar{\Bell}_1,\ldots,\widebar{\Bell}_\kappa,\mathbf{x}_1,\ldots,\mathbf{x}_{\kappa-1}\right)$ is
		\begin{align}\label{eq:jointdistribution}
			\chi_K(\Lj,x[\kappa-1]):=\prod_{j=1}^{\kappa}m_j\exp\left(-m_j\widebar{\ell}_j\right)\cdot\sqrt{\frac{\mathbb{d}_K}{(2\pi)^{\kappa-1}}}\exp\left(-\frac{1}{2} {}^t x\Sigma_K^{-1}x\right).
		\end{align}
	\end{Theoreme}
	
	\begin{Remarque}
		Note that we do not use the variable $\displaystyle\mathbf{x}_\kappa^{(n)}:=({\mathbf{s}^{(n)}_\kappa-ng_\kappa})/{\sqrt{n}}$ since its data are contained within $\mathbf{x}^{(n)}[\kappa-1]$ as $\mathbf{x}_\kappa^{(n)}=-\sum_{j=1}^{\kappa-1}\mathbf{x}_j^{(n)}$.
		
		In fact this theorem states that asymptotically, $\mathbf{x}^{(n)}[\kappa-1]$ has a density with respect to the Lebesgue measure on $\RR^{\kappa-1}.$ But we can also deduce that $\left(\widebar{\Bell}_1^{(n)},\ldots,\widebar{\Bell}_\kappa^{(n)},\mathbf{x}_1^{(n)},\ldots,\mathbf{x}_{\kappa}^{(n)}\right)$ converges in distribution and that for all bounded continuous test function $\psi:\RR_{+}^\kappa\times\RR^{\kappa}\to\RR$,  
		$$\mathbb{E}\left[\psi\left(\widebar{\Bell}^{(n)}[\kappa],\mathbf{x}^{(n)}[\kappa]\right)\right]\cvg \mathbb{E}\left[\psi\left(\widebar{\Bell}[\kappa],\mathbf{x}_1,\ldots,\mathbf{x}_{\kappa-1},-\mathbf{x}_1-\ldots-\mathbf{x}_{\kappa-1}\right)\right].$$
	\end{Remarque}
	
	Before we get to the proof of this theorem, we recall the following (simple but) powerful lemma that we also used in \cite{Morin}[Lemma 4.3] (and which is proven there).
	
	\begin{Lemma}\label{lemdebase}
		Let $(\chi_n)_{n\in\NN}$ be a sequence of nonnegative measurable functions on $\RR^d$. Assume that for all $\varepsilon>0$, there exists a compact set $K_\varepsilon$ such that for all $n$ large enough, $\int_{K_\varepsilon^c}\chi_n<\varepsilon$ (where $K_\varepsilon^c$ is the complement of $K_\varepsilon$ in $\RR^d$), and that $\chi_n$ uniformly converges on all compact sets of $\RR^d$ towards a density $\chi$ (with respect to the Lebesgue measure on $\RR^d$). Then, there exists a sequence $(\alpha_n)_{n\in\NN}$ such that for $n$ large enough (for small values of $n$, $\chi_n$ could be zero), $\frac{1}{\alpha_n}\chi_n$ is a density and $\alpha_n\underset{n\to+\infty}{\longrightarrow}1.$
	\end{Lemma}
	
	\begin{proof}[Proof of \Cref{thm:commeavant}]
		The successive steps of this proof are generalizations of those of \cite{Morin}[Theorem 4.1.]. 		
		The proof should be carried out in two steps, but only the first one is detailed, as the second is pretty much identical to the that given in \cite{Morin}[Theorem 4.1.]. These steps are as follows:
		\begin{enumerate}
			\item We show the uniform convergence on compact sets of the ``joint density'' $f^{(n)}_K$ of the pair $\left(\widebar{\Bell}^{(n)}[\kappa],\mathbf{x}^{(n)}[\kappa-1]\right)$. More exactly, we show the uniform convergence on compact sets of a density $\chi^{(n)}_K$ introduced in (\ref{gn}) associated to these random variables to $\chi_K$.
			\item At this point we should give an argument of uniform integrability for this limit to apply \Cref{lemdebase} and conclude. This part will be skipped as it is only a rewriting of the second step of \cite{Morin}[Theorem 4.1.].
		\end{enumerate}
		
		We detail {\bf Step 1:} Let $\psi:\RR_{+}^\kappa\times\RR^{\kappa-1}\to\RR$ be a bounded continuous test function and let us pass to the limit in the expectation 
		\begin{align}\label{eq:exp}
			\mathbb{E}\left[\psi\left(\widebar{\Bell}^{(n)}[\kappa],\mathbf{x}^{(n)}[\kappa-1]\right)\right]&=\sum_{\ssk\in\NN_K(n)}\int_{\RR^\kappa}\psi\left(n\ell[\kappa],\frac{s-ng}{\sqrt{n}}[\kappa-1]\right)f^{(n)}_{K}\left(\ell[\kappa],\ssk\right)\mathrm{d}\ell[\kappa],
		\end{align}
		where the joint distribution $f^{(n)}_{K}$ of the couple $(\lk,\sk)$ is given in \Cref{thm:distri2}.
		
		We want to perform substitutions of the type $\widebar{\ell}^{(n)}_j=n\ell_j$ and $\displaystyle x^{(n)}_j={(s_j-ng_j)}{/\sqrt{n}}$ in the right-hand side of \eqref{eq:exp}, so to prove that $\widebar{\Bell}^{(n)}_j$ converges to an exponentially distributed variable and $\mathbf{x}^{(n)}[\kappa-1]$ to a Gaussian vector.
		
		To make the support of the Gaussian vector apparent, let us turn our sum over $\ssk\in\NN_K(n)$ into an integral. We set $s_j=\flr{q_j}$ for all $j\in\entk$, where $q[\kappa-1]$ browses the set $\bNkn:=\left\{q[\kappa-1],\text{ with }q_j>0,\tst\sum_{j=1}^{\kappa-1} \flr{q_j}\leq n\right\}$, with convention set to $\flr{q_\kappa}:=n-\sum_{j=1}^{\kappa-1} \flr{q_j}$ (notice that there is no integration with respect to $q_\kappa$).
		
		With \eqref{eq:rewritingjoint}, we get to :
		\begin{multline*}
			\mathbb{E}\left[\psi\left(\widebar{\Bell}^{(n)}[\kappa],\mathbf{x}^{(n)}[\kappa-1]\right)\right]=\Faktor(n)\int_{\bNkn}\int_{\mathcal{L}_K}\psi\left(n\ell[\kappa],\frac{\flr{q}-ng}{\sqrt{n}}[\kappa-1]\right)\mathbb{M}(\lfloor q[\kappa]\rfloor)\\
			\times\prod_{j=1}^{\kappa} \left(1-\frac{\cl_j(\Lj)}{r_j}\right)^{\lfloor q_j\rfloor+\lfloor q_{\wj}\rfloor-1}\cdot\mathrm{d}\ell[\kappa]\mathrm{d}q[\kappa-1],
		\end{multline*}
		but let us rather consider the expectation
		\begin{multline}\label{eq:cool}
			\mathbb{E}\left[\psi\left(\widebar{\Bell}^{(n)}[\kappa],\widebar{\mathbf{x}}^{(n)}[\kappa-1]\right)\right]=\Faktor(n)\int_{\bNkn}\int_{\mathcal{L}_K}\psi\left(n\ell[\kappa],\frac{{q}-ng}{\sqrt{n}}[\kappa-1]\right)\mathbb{M}(\lfloor q[\kappa]\rfloor)\\
			\times\prod_{j=1}^{\kappa} \left(1-\frac{\cl_j(\Lj)}{r_j}\right)^{\lfloor q_j\rfloor+\lfloor q_{\wj}\rfloor-1}\cdot\mathrm{d}\ell[\kappa]\mathrm{d}q[\kappa-1],
		\end{multline} (we withdrew the floor function in $\psi$), where for all $j\in\{1,\ldots,\kappa-1\}$, we set $\widebar{\mathbf{x}}^{(n)}_j:=\mathbf{x}^{(n)}_j+\mathbf{U}_j/\sqrt{n}$, with $\mathbf{U}_j$ a r.v. uniformly distributed in $\zerun$.
		
		We replaced a sum by an integral and this amounts to representing a discrete random variable by a continuous one, \ie if $X$ has a discrete law, $\mathbb{P}(X=k)=p_k,k\in\ZZ$, then $X\sur{=}{(d)}\lfloor X+U\rfloor$, where $U$ is uniform in $\zerun$. Then, 
		\[\sum_{k\in\ZZ} p_kf(k)=\int_\RR f(\lfloor x\rfloor) p_{\lfloor x\rfloor}\mathrm{d}x.\]
		We are going to prove first that $\mathbb{E}\left[\psi\left(\widebar{\Bell}^{(n)}[\kappa],\widebar{\mathbf{x}}^{(n)}[\kappa-1]\right)\right]$ converges to deduce that its counterpart $\mathbb{E}\left[\psi\left(\widebar{\Bell}^{(n)}[\kappa],{\mathbf{x}}^{(n)}[\kappa-1]\right)\right]$ converges as well, to the same limit.
		
		We can finally make our substitutions $\widebar{\ell}_j:=\widebar{\ell}^{(n)}_j=n\ell_j$ for all $j\in\entk$ and $\displaystyle x_j:=x^{(n)}_j={(q_j-ng_j)}/{\sqrt{n}}$ for all $j\in\ent{1}{\kappa-1}$ in \eqref{eq:cool} where we have hidden the dependence on $n$ to unburden our equations. We reach
		\begin{align}\label{jpp1}
			\mathbb{E}\left[\psi\left(\widebar{\Bell}^{(n)}[\kappa],\widebar{\mathbf{x}}^{(n)}[\kappa-1]\right)\right]=\int_{\RR^{\kappa-1}}\int_{\RR_+^\kappa}\psi(\widebar{\ell}[\kappa],x[\kappa-1])\chi^{(n)}_K(\widebar{\ell}[\kappa],x[\kappa-1])\mathrm{d}\widebar{\ell}[\kappa]\mathrm{d}x[\kappa-1],
		\end{align}	
		where for all $n\geq 3,$ $\chi^{(n)}_K$ stands for the joint distribution of the pair $\left(\widebar{\Bell}^{(n)}[\kappa],\widebar{\mathbf{x}}^{(n)}[\kappa-1]\right)$. The function $\chi^{(n)}_K$ can be decomposed as follows
		\begin{align}\label{gn}
			\chi^{(n)}_K\left(\widebar{\ell}[\kappa],x[\kappa-1]\right):=\omega(n,K)~h^{(1)}_n(\widebar{\ell}[\kappa],x[\kappa-1])~h^{(2)}_n(x[\kappa-1]),
		\end{align}
		with 
		\begin{multline}\label{equ7}
			\omega(n,K)=\left[\prod_{j=1}^\kappa\frac{f_j+f_{{j-1}}}{\sin(\theta_j)r_j}\cdot\frac{n!~e^{3n}(\bar{f})^{3n-\kappa}}{\ptK~n^{3n-\kappa}}\right]\left[\frac{1}{n^\kappa\sqrt{(2\pi)^{\kappa+1}\mathbb{d}_K}\prod_{j=1}^\kappa\sqrt{g_j(g_j+g_{\wj})}}\right]\\\times\left[\prod_{j=1}^\kappa\frac1{m_j}\right]\frac{1}{n^{(\kappa+1)/2}},
		\end{multline} which can be rewritten as
		\begin{align}
			\omega(n,K)&=\frac{1}{\sqrt{(2\pi)^{\kappa+1}\mathbb{d}_K}}\cdot\left[\prod_{j=1}^\kappa\sqrt{\frac{f_j+f_{{j-1}}}{f_j}}\frac{1}{m_j\sin(\theta_j)r_j}\right]\cdot\frac{n!~e^{3n}(\bar{f})^{3n}}{\ptK~n^{3n}}\frac{1}{n^{(\kappa+1)/2}},
		\end{align}
		and
		\begin{align}
			h^{(1)}_n\left(\widebar{\ell}[\kappa],x[\kappa-1]\right)&=\mathbb{1}_{\widebar{\ell}[\kappa] \in n\mathcal{L}_K}~\prod_{j=1}^\kappa m_j\left(1-\frac{\cl_j(\Lj)}{n\cdot r_j}\right)^{d^{(2)}_j(x[\kappa-1])},\\
			h^{(2)}_n\left(x[\kappa-1]\right)&=\sqrt{\frac{\mathbb{d}_K}{(2\pi)^{\kappa-1}}}\prod_{j=1}^\kappa (2\pi n)\sqrt{g_j(g_j+g_{\wj})}\cdot\mathbb{M}(d^{(1)}\left(x[\kappa-1]\right)),\\[2pt]
			\intertext{where, with convention $x_\kappa:=-\sum_{j=1}^{\kappa-1}x_j$, we set}\label{eq:conventions1}
			d^{(1)}_j\left(x[\kappa-1]\right)&=\lfloor ng_j + \sqrt{n}x_j \rfloor,\text{ for all }j\in\entk,\\[10pt]
			\label{eq:conventions2}d^{(2)}_j\left(x[\kappa-1]\right)&=\lfloor ng_j + \sqrt{n}x_j \rfloor+\lfloor ng_{j-1} + \sqrt{n}x_{j-1} \rfloor-1,\text{ for all }j\in\entk.
		\end{align}
		
		Notice here that there is a small approximation with this convention : indeed, with this change of variable, the natural definition of $d^{(1)}_\kappa\left(x[\kappa-1]\right)$ in the integral corresponding to the expectation $\mathbb{E}\left[\psi\left(\widebar{\Bell}^{(n)}[\kappa],\widebar{\mathbf{x}}^{(n)}[\kappa-1]\right)\right]$ should be $d^{(1)}_\kappa\left(x[\kappa-1]\right):=n-{\sum_{j=1}^{\kappa-1}\flr{\sqrt{n}x_j+ng_j}}$ to be consistent (the same problem appears for $d^{(2)}_{j}=\flr{q_j}+\flr{q_{j+1}}-1$ when $j\in\{\kappa-1,\kappa\})$. This means that the definitions \eqref{eq:conventions1} and \eqref{eq:conventions2} correspond to the expectation $\mathbb{E}\left[\psi\left(\widebar{\Bell}^{(n)}[\kappa],\widetilde{\mathbf{x}}^{(n)}[\kappa-1]\right)\right]$ for some variables $\widetilde{\mathbf{x}}^{(n)}[\kappa-1]$ close to $\widebar{\mathbf{x}}^{(n)}[\kappa-1]$.  But it is clear that with
		\[\frac{n-{\sum_{j=1}^{\kappa-1}\flr{\sqrt{n}x_j+ng_j}}}{\flr{\sqrt{n}x_\kappa+ng_\kappa}}\cvg 1,\] we can easily prove that $\abso{\widebar{\mathbf{x}}^{(n)}[\kappa-1]-\widetilde{\mathbf{x}}^{(n)}[\kappa-1]}\proba 0$ and thus by Slutsky's Lemma, the limits in distribution of $\left(\widebar{\Bell}^{(n)}[\kappa],\widebar{\mathbf{x}}^{(n)}[\kappa-1]\right)$ and $\left(\widebar{\Bell}^{(n)}[\kappa],\widetilde{\mathbf{x}}^{(n)}[\kappa-1]\right)$ are the same.
		
		We have arranged the factors so that, as we will see, $h^{(1)}_n$ and $h^{(2)}_n$ converge to some probability densities (and since a density has integral 1, this will be a crucial point for the conclusion).
		
		Note first that there exists $\eta>0$ such that $[0,\eta]^\kappa\subset \mathcal{L}_K$ and thus, we have $n\mathcal{L}_\kappa\cvg \RR_{+}^\kappa$. Then for every compact $H\subset \RR_{+}^\kappa$, and for all $\varepsilon>0,$ there exists $n_0\in\NN$ such that for all $n\geq n_0,$ $H\subset n\mathcal{L}_K$, \ie $\norm{\indic{{\widebar{\ell}[\kappa]\in n\mathcal{L}_K}}-1}=0<\varepsilon$, so that the map $\widebar{\ell}[\kappa]\mapsto\indic{{\widebar{\ell}[\kappa]\in n\mathcal{L}_K}}$ converges uniformly to the constant function $1$ on every compact set of $\RR_+^\kappa$. Now by the standard approximation $$\left(1-\frac{a}{n}\right)^{nb}\underset{n\to+\infty}{\longrightarrow}e^{-ab}$$ uniformly for $(a,b)$ on every compact set of $\RR^2$, and using additionnaly the fact that $\mathbb{1}_{\widebar{\ell}[\kappa]\in n\mathcal{L}_K}\underset{n\to+\infty}{\longrightarrow}1$ uniformly on every compact set of $\RR_+^\kappa$, we get that $h^{(1)}_n$ 
		converges uniformly on every compact set of $\RR_{+}^\kappa\times\RR^{\kappa-1}$ towards $h^{(1)}$, with 
		\begin{align}\label{equ8}
			h^{(1)}\left(\widebar{\ell}[\kappa],x[\kappa-1]\right)=\prod_{j=1}^{\kappa}m_j\exp\left(-m_j\widebar{\ell}_j\right).
		\end{align}

		Now thanks to the local limit theorem for Poisson r.v. \cite[Theorem VII.1.1]{petrov1975sums}, $h_n^{(2)}$ converges uniformly towards $h^{(2)}$, where $h^{(2)}$ is defined for $x[\kappa-1]\in\RR^{\kappa-1},$ as
		\begin{align}\label{eq:jpp8}
			h^{(2)}(x[\kappa-1]):=\sqrt{\frac{\mathbb{d}_K}{(2\pi)^{\kappa-1}}}\prod_{j=1}^\kappa\exp\left(-\frac{1}{2} \left[\frac{x_j^2}{g_j}+\frac{(x_j+x_{\wj})^2}{g_j+g_{\wj}}\right]\right).
		\end{align}
		
		The map $h^{(2)}$ can be rewritten as 
		\[h^{(2)}(x[\kappa-1])=\sqrt{\frac{\mathbb{d}_K}{(2\pi)^{\kappa-1}}}\exp\left(-\frac{1}{2} {}^t x\Sigma_K^{-1}x\right),\]
		where $\Sigma_K^{-1}$ given in \Cref{thm:tangency}, corresponds to the matrix associated to the exponential product in \eqref{eq:jpp8}, and $\mathbb{d}_K=\det\left(\Sigma_K^{-1}\right)$ is its determinant.
		\par We have established the following uniform convergence on every compact set of $\RR_+^{\kappa}\times\RR^{\kappa-1}$ :
		\begin{align}\label{equ2}
			\frac{1}{\omega(n,K)}\chi^{(n)}_K(\widebar{\ell}[\kappa],x[\kappa-1])\underset{n\to +\infty}{\longrightarrow} \chi_{K}(\widebar{\ell}[\kappa],x[\kappa-1]).
		\end{align}
		This concludes the first step of our proof. \\
		As said in introduction, the uniform integrability is the same as in \cite{Morin}[Theorem 4.1.], so we skip this step and send the reader to \cite{Morin} for the remaining details. 
		However, the conclusion of this proof is the same !
		
		This proves that $\chi^{(n)}_{K}$ converges pointwise to $\chi_{K}$, or by definition, that $(\widebar{\Bell}^{(n)}[\kappa],\widebar{\mathbf{x}}^{(n)}[\kappa-1])\dd (\widebar{\Bell}[\kappa],\mathbf{x}[\kappa-1])$. Now, since $\abso{\widebar{\mathbf{x}}^{(n)}[\kappa-1]-\mathbf{x}^{(n)}[\kappa-1]}\proba0$, then by Slutsky's Lemma, \[(\widebar{\Bell}^{(n)}[\kappa],{\mathbf{x}}^{(n)}[\kappa-1])\dd (\widebar{\Bell}[\kappa],\mathbf{x}[\kappa-1]).\]
		
		In the same time, we obtain that $\omega(n,K)\underset{n\to+\infty}{\longrightarrow}1$, which, with Stirling's formula, leads to
		\begin{align}\label{eq:moltoimportante}
			\ptK\underset{n\to+\infty}{\sim} \frac{1}{(2\pi)^{\kappa/2}\sqrt{\mathbb{d}_K}}\cdot\left[\prod_{j=1}^\kappa\sqrt{\frac{f_j+f_{{j-1}}}{f_j}}\frac{1}{m_j\sin(\theta_j)r_j}\right]\cdot\frac{e^{2n}(\bar{f})^{3n}}{n^{2n+\kappa/2}}.
		\end{align}
		\Cref{thm:barany1} shows in particular that $\displaystyle (\bar{f})^{3n}=\frac{\AP^*(K)^{3n}}{4^n}$, and that $\displaystyle \frac{f_j}{f_j+f_{{j-1}}}=w_j$.
		
	\end{proof}

	\section{General case}
	
	This section is dedicated to the proof of \Cref{thm:global}, where we exhibit the link between $\mathbf{P}$ and the set $\polytau$. To do so, we need to understand how the limit shape is structured in a convex polygon. The next results are stated in this regard.
	
	\begin{Lemma}\label{lem:tangency}
		Fix $\kappa\geq3.$ For any $K\in\polyk$, the set $\Dom{K}$ is tangent to $K$ at $m$ points for some $m\in\ent{3}{\kappa}$, and its boundary $\partial\Dom{K}$ is composed of parabola arcs between these tangency points.
	\end{Lemma}
	
	\begin{proof}
		If $\Dom{K}$ were tangent only at $m<3$ points of $\partial K$, a simple figure shows that there would be some room to set in $K$ a set bigger than $\Dom{K}$ obtained by applying a dilatation and a translation. 
	\end{proof}
	
	\begin{Proposition}[\cite{barany1}]\label{prop:bara}
		Let $K\in\polyk$. The boundary of the set $\Dom{K}$ is composed of $m$ pieces of parabola arc (with $3\leq m\leq \kappa$ by \Cref{lem:tangency}) whose tangents at the endpoints are tangent to $\Dom{K}$ as well.
	\end{Proposition}
	
	\paragraph{Notation.} As a consequence of \Cref{prop:bara}, we denote by $\mathcal{S}(K)$ the set of compact convex sets of $\RR^2$ such that $S\in\mathcal{S}(K)$ if and only if $\Dom{K}\subset S\subset K$. In particular, if we denote $\mathsf{p}_1,\ldots,\mathsf{p}_m$ the tangency points of $\Dom{K}$ with $K$, $\Dom{K}$ is also tangent to $S$ at all $\mathsf{p}_1,\ldots,\mathsf{p}_m$.

	\begin{Lemma}\label{lem:reduction}
		Let $K\in\mathbf{P}$. Let $S\in\mathcal{S}(K)$ and let $\zn$ be taken under $\mathsf{Q}^{(n)}_{S}$ for some $n\geq0$. Then for all $\varepsilon>0$, we have
		\begin{align}\label{eq:limitst2}
			\P\left(d_H\left(\CH\left(\zn\right),\Dom{K}\right)>\varepsilon\right)\cvg0.
		\end{align}
	\end{Lemma}
	This means that for any set $S\in\mathcal{S}(K)$, the convex hull of $\zn$ taken under $\mathsf{Q}^{(n)}_{S}$ converges a.s. for the Hausdorff distance to the same deterministic limit as if the points were taken under $\mathsf{Q}^{(n)}_{K}$, which is nothing but the set $\Dom{K}$.
	\begin{proof}
		By B\'ar\'any's limit shape theorem, the deterministic limit $\Dom{K}$ is characterized as the unique set such that \[\AP^*(K)=\AP(\Dom{K})=\max_{\substack{S \text{convex} \\ S\subset K}} \AP(S).\]
		Now, for all $S\in\mathcal{S}(K)$, we have $\AP(\Dom{K})\leq\AP^*(S)\leq \AP^*(K)$.
	\end{proof}
	
	\subsection*{The general result: proofs of \Cref{lem:rep} and \Cref{thm:global}:}
	
	The idea of these proofs is to find sort of a ``converse'' to the previous lemma. Indeed, \Cref{lem:reduction} allows one to shrink a convex polygon while keeping its limit shape intact, as long as the limit shape is not shrunk at the same time. Is there any way we could ``increase'' (in some sense) a convex polygon, while keeping its limit shape intact ?
	
	\begin{proof}
		The case $\kappa=3$ does not need \Cref{thm:global}. Indeed, the limit shape in an equilateral triangle is tangent to all three sides, and since every triangle is obtained by an affine map of an equilateral triangle, the limit shape of a triangle is tangent to all sides as well.
		
		Fix $\kappa\geq3$ and let us rather assume that $K$ is in $\mathbf{P}_{\kappa+1}$, and that $\Dom{K}$ is tangent to exactly $\kappa$ sides of $K$ (that is, all sides except one). 
		{\it We focus on this case for now, and explain at the end why the same arguments as those developed in the following stay valid for any $K$ in $\mathbf{P}_{\kappa+j}$, $j\geq1$, having $\Dom{K}$ tangent to exactly $\kappa$ sides of $K$.}
		
		Denote by $\vec{n}_1,\ldots,\vec{n}_{\kappa+1}$ the successive sides of $K$, viewed as vectors (or directions) in this proof. By hypothesis, the set $\Dom{K}$ is tangent to $K$ on $\kappa$ sides and thus we assume that the $(\kappa+1)^{th}$ side is the missing one, as in \Cref{chap3:fig18}. Consider the $\kappa$ straight lines directed by $\left(\vec{n}_i\right)_{1\leq i\leq\kappa}$; they delineate a convex polygon $\Ktau\in\polyk$ (with $\kappa$ sides) that contains $K$. We let $\mathsf{X}$ be the vertex of $\Ktau$ being the intersection point between the lines of directions $\vec{n}_\kappa$ and $\vec{n}_1.$
		
		\begin{figure}[hbtp]
			\centering
			\input{Figure_7bis.txt}
			\caption{In cyan, an example of polygon $K$ whose limit shape (the boundary of $\Dom{K}$, in red) is not tangent to one side. The polygon delineated by the dotted black lines is the polygon $\Ktau$ : the boundary of $\Dom{K}$ is now tangent to every side of $\Ktau$. We will show that $\Dom{K}=\Dom{\Ktau}$.} \label{chap3:fig18}
		\end{figure}

		Let us prove that $\Ktau$ satisfies \eqref{jpp8}, \eqref{jpp10} and \eqref{jpp9}. To do so, consider the family of convex polygons $(K_{t})_{0\leq t\leq1}$ defined as 
		\begin{enumerate}
			\item[{\it (i)}] $K_0=K,$
			\item[{\it (ii)}] $K_1=\Ktau$,
			\item[{\it (iii)}] for $t>0$, $K_t$ is the polygon $K_0$ where we ``pushed back'' at distance $t$ the $(\kappa+1)^{th}$ side, with direction $\vec{n}_{\kappa+1}$ (which is thus at distance $1-t$ from $\mathsf{X}$).\footnote{More precisely, this is the intersection between $\Ktau$ and the halfplane delineated by the straight line with direction $\vec{n}_{\kappa+1}$ at distance $1-t$ from $X$.} Of course, we have for all $t\in\zerun$,  $K_t\subset\Ktau$.
		\end{enumerate}

		We illustrate this definition in \Cref{fig:Ktau}. Note that, with this definition of $(K_t)$ we have for all $t\in\zerun$,
		\begin{align}\label{eq:continuityK}
			d_H(K_s,K_t)\underset{s\to t}{\longrightarrow} 0.
		\end{align}

		\begin{figure}[hbtp]
			\centering
			\begin{minipage}{0.45\textwidth}
				\centering
				\input{Figure_8.txt}
			\end{minipage}
			\hspace{1ex}
			\begin{minipage}{0.45\textwidth}
				\centering
				\input{Figure_7.txt}
			\end{minipage}
			\caption{The polygon $K_t$, drawn in blue at times $t=0$ and $t=1$, and its limit shape in red.}
			\label{fig:Ktau}
		\end{figure}
		
		\vspace{0.2cm}
		We are going to prove that for all $t\in\zerun,$ we have $\Dom{K_t}=\Dom{K}$. To do so, we prove that the map $F:t\mapsto \AP^*(K_t)$ is continuous, and then we will deduce that it must be constant on $\zerun$.
		
		\paragraph{1. The map $F:t\mapsto \AP^*(K_t)$ is continuous.}
		First, notice that since the affine perimeter is an upper-semicontinuous map (a proof of this claim can be found in \cite{lutwak}[Proposition 4.4]), so is $F$. In addition, $F$ is non-decreasing on $\zerun$ by definition.
		
		Let us assume now that $F$ is not continuous, and denote by $t^*>0$ the moment of a jump of $F$, so that $\limsup_{t\to t*}F(t)<F(t^*)$. Let then $\alpha$ be a real value in the open interval $(\limsup_{t\to t*}F(t),F(t^*))$.
		
		There exists a dilatation by some $0<\lambda<1$ such that $\AP(\lambda\Dom{K_{t^*}})=\lambda^{2/3}F(t^*)=\alpha$. Since $\lambda\Dom{K_{t^*}}$ is just a ``shrunk'' version of $\Dom{K_{t^*}}$, there exists $0<\tilde{t}<t^*$ with $\lambda\Dom{K_{t^*}}\subset K_{\tilde{t}}$. This means in particular that $F(\tilde{t})\geq \alpha>\limsup_{t\to t*}F(t)$ and then $F(\tilde{t})=F(t^*)$ because $F$ has no image in the interval $(\limsup_{t\to t*}F(t),F(t^*))$. Since $F$ is nondecreasing, then $F$ is constant on the interval $[\tilde{t},t^*]$: this is absurd since $\limsup_{t\to t*}F(t)<F(t^*)$.
		
		\paragraph{2. The map $F:t\mapsto \AP^*(K_t)$ is constant.}
		Assume it is not : there exists $t>0$ such that we have $F(t)>F(0)$. By continuity, there exists $s\geq0$ such that $[s,t]\subseteq [0,t]$ and $F([s,t])=[F(0),F(t)]$, and $F$ is constant on $[0,s]$.
		
		For $n$ large enough, there exists a sequence $(s_n)\in [s,t]$ with $F(s_n)=F(s)+1/n$. Consider the sequence $(\mathsf{D}_{s_n})_n:=(\Dom{K_{s_n}})_n$, it is a sequence of compact convex sets that are all subsets of $K_1$, and since the set of compact convex subsets of $K_1$ is compact for $d_H$ (see \cite{istratescu2001fixed}), there exists a subsequence $n_k$ such that $(\mathsf{D}_{s_{n_k}})$ converges to a convex set $\bar{\mathsf{D}}$ for the Hausdorff distance.
		
		By upper-semicontinuity of the affine perimeter, we have $\AP(\bar{\mathsf{D}})\geq \limsup_{k\to+\infty} F(s_{n_k})=F(s).$ In addition, since $\bar{\mathsf{D}}=\lim_{k\to+\infty} \mathsf{D}_{s_{n_k}} $ and $\mathsf{D}_{s_{n_k}}\subseteq K_{s_{n_k}}$, we have $\bar{\mathsf{D}}\subseteq K_s$ (since $d_H(K_{s_{n_k}},K_s)\underset{k\to+\infty}{\longrightarrow}0$ by \eqref{eq:continuityK}). By unicity, this means $\bar{\mathsf{D}}=\Dom{K_s}=\Dom{K_0}.$
		
		Let now $\mathscr{L}_s$ be the segment corresponding to the $(\kappa+1)^{th}$ side of $K_s$, the one we added  with direction $\vec{n}_{\kappa+1}$ at distance $(1-s)$ of $\mathsf{X}$. 
		
		All sets $\mathsf{D}_{s_n}$ (for $n$ large enough) intersect with $K_s$ (because $F(s_n)>F(s)$ and thus, $\mathsf{D}_{s_n}\subsetneq K_s$ would go against the fact that $\mathsf{D}_s$ realizes the maximum of the affine perimeter among convex subsets of $K_s$), meaning we have $d_H(\mathsf{D}_{s_{n_k}},\mathscr{L}_s)=0$; hence, when $k\to+\infty$, we get $d_H(\bar{\mathsf{D}},\mathscr{L}_s)=0$ by continuity. But since $\Dom{K_0}=\Dom{K_s}=\bar{\mathsf{D}}\subset K_s$, this means $\Dom{K_0}$ must intersect (or at least be tangent with) $\mathscr{L}_s$, which is impossible because $d_H(\Dom{K_0},\mathscr{L}_s)\geq d_H(\Dom{K_0},\mathscr{L}_0)$ and we assumed that $d_H(\Dom{K_0},\mathscr{L}_0)>0$. 	

		This proves that $F$ is constant, hence $\Dom{K_{1}}=\Dom{K_{0}}=\Dom{K_{\mathcal{T}}}$ and this proves \eqref{jpp8} and \eqref{jpp10} simultaneously. 
		
		To prove \eqref{jpp9}, the properties $\Dom{K}=\Dom{K_\mathcal{T}}$ and $K\subset \Ktau$, yield that for $\varepsilon>0$ small enough, there exists $N\in\NN$ such that for all $n\geq N$, we have $S_n(K,\varepsilon)=S_n(\Ktau,\varepsilon)$ (see notation \eqref{not:sng}). The last conclusion stated in \Cref{thm:lst} enables this conclusion.
		
		\paragraph{3. A remark about how we may replace $\mathscr{L}_s$.}
		In the second point of this proof, we use the fact that $d_H(\mathsf{D}_{s_{n_k}},\mathscr{L}_s)=0$ and we take the limit on $k$ to obtain $d_H(\bar{\mathsf{D}},\mathscr{L}_s)=0$. In this argument, we do not use the fact that $\mathscr{L}_s$ is a segment.
		Assume now that instead of the $(\kappa+1)^{th}$ side of $K_0$ \ie the segment $\mathscr{L}_0$, we have a concave closed shape $\mathscr{S}_0$ joining the extremities of the $1^{st}$ and $\kappa^{th}$ sides of $K_0$ forming a convex set $S_0$, as long as $S_0\subseteq K_1$. We still have $\Dom{K_0}\subset S_0$ and $d_H(\Dom{K_0},\mathscr{S}_0)>0$. We give two examples of sets $S_0$ in \Cref{fig:S0}.
		
		Let us further define a family of increasing convex sets $(S_t)_{t\in\zerun}$ such that :
		\begin{enumerate}
			\item[{\it (i)}] $S_0$ is the set we just defined,
			\item[{\it (ii)}] The function $t\mapsto S_t$ is non-decreasing for the inclusion order, continuous for $d_H$ and satisfies $S_1 = K_1$.
		\end{enumerate}
		
		Under this condition, this means that for all $t\in[0,1)$, the boundary of the set $S_t$ is composed of the $\kappa$ polygonal lines forming the closed set $\partial K_1\cap\partial S_t $, which are joined together by a concave closed shape $\mathscr{S}_t$ lying in the interior of $K_1$.
		
		\begin{figure}[hbtp]
			\centering
			
			\begin{minipage}{0.45\textwidth}
				\centering
				\input{Figure_S1.txt}
			\end{minipage}
			\hspace{1ex}
			\begin{minipage}{0.45\textwidth}
				\centering
				\input{Figure_S2.txt}
			\end{minipage}
			
			\caption{Two kinds of convex body $S_0$, drawn in blue at times $t=0$, and their same limit shape in red, which is $\Dom{K_0}$. On the left, the concave shape $\mathscr{S}_0$ is a smooth map, while on the right it is composed of several polygonal lines.}
			\label{fig:S0}
		\end{figure}
		
		\vspace{0.2cm}
		Had we used this family instead of $(K_t)$ in the proof of point $2.$, this would have just changed the closed domain $\mathscr{L}_s$ into $\mathscr{S}_s$ and the rest of the proof would have remained the same. In the end, we would get a slightly better result as this leads to the fact that $\Dom{K_0}=\Dom{S_t}=\Dom{K_1}$ for all $t\in\zerun$.

		\paragraph{4. The case $\mathbf{P}_{\kappa+j}$.}
		We can now generalize this reasoning to the case where $K$ is in $\mathbf{P}_{\kappa+j}$, $j\geq2$, and $\Dom{K}$ is tangent to exactly $\kappa$ sides of $K$. Assume first that the $j$ sides to which $\Dom{K}$ is not tangent are consecutive.
		Denote once more by $\vec{n}_1,\ldots,\vec{n}_{\kappa+j}$ the successive sides of $K$, and assume first that the last $j$ sides are the ``missing'' ones. It means that there is an arc of parabola joining the $n_\kappa^{th}$ and $n_1^{st}$ sides, to which it is also tangent. In particular, the angle between the $n_\kappa^{th}$ and $n_1^{st}$ side must be closing, or the parabola could not be tangent to these sides. Consider the $\kappa$ straight lines directed by $\left(\vec{n}_i\right)_{1\leq i\leq\kappa}$; they delineate a convex polygon $\Ktau\in\polyk$ (with $\kappa$ sides) that contains $K$. 
		
		This means that we can construct a family $(S_t)$ satisfying conditions {\it (i)} and {\it (ii)} of point $3.$ (just as on the right of \Cref{fig:S0}), where the concave closed shape $\mathscr{S}_0$ comprises the last $j$ ``missing'' sides lying in the interior of $S_1:=\Ktau$.
		In this case, we have $S_0:=K$; by the previous point we obtain $\Dom{\Ktau}=\Dom{K}$, which proves this case of consecutive ``missing'' sides.
				
		If the $j$ sides to which $\Dom{K}$ is not tangent are not consecutive, we can still denote by $\vec{n}_1,\ldots,\vec{n}_{\kappa+j}$ the directions of the successive sides of $K$, and assume that the last $i$ sides, with $1\leq i<j$ are ``missing'' sides. Consider the $\kappa+j-i$ straight lines directed by $\left(\vec{n}_q\right)_{1\leq q\leq\kappa+j-i}$; they delineate a convex polygon $K_1$ with $\kappa+j-i$ sides that contains $K$. The exact same arguments as before allow us to prove that $K\in\mathcal{S}(K_1)$ and \[{\sf Area}(K)^n~\P_{K}(n)\underset{n\to\infty}{\sim} {\sf Area}(K_1)^n~\P_{K_1}(n).\] Step by step, treating each bunch of consecutive missing sides one after the other, there will exist a finite sequence of convex polygons $K_1\subset\ldots\subset K_k= \Ktau$, with $\Ktau\in\polyktau$ (with $\kappa$ sides) and $K_p\in\mathcal{S}(K_{p+1})$ such that \[{\sf Area}(K_p)^n~\P_{K_p}(n)\underset{n\to\infty}{\sim} {\sf Area}(K_{p+1})^n~\P_{K_{p+1}}(n)\] for all $p\in\ent{1}{k-1}$, so that the last step amounts to proving the previous case. That proves \eqref{jpp9}.

	\end{proof}
	
	\section{Determining $\Dom{K}$}\label{sec4}
	
	Let $K\in\polyk$. Determining $\Dom{K}$ out of its characterization as $\displaystyle\Dom{K}=\arg \max_{\substack{S \text{convex set} \\ S\subset K}}\AP(S)$ is not trivial, hence it is neither trivial to understand whether $K$ is in $\polyktau$ or not. We present in this section an algorithm that allows us to solve this question. It relies on the previous proof of \Cref{thm:global} and aims at finding the convex set $\Ktau$.\par
	
	We consider the straight lines $(d_1),\ldots,(d_\kappa)$ that contain, respectively, the sides $r_1,\ldots,r_\kappa$ of $K$.\\
	
	\begin{Notation}
		For all $I\subset\entk$ such that $\abso{I}\geq3$, consider the set $K_I\subset \RR^2$ defined as the set delineated by the $\abso{I}$ straight lines $(d_i)_{i\in I}$. Set then 
		\[{\sf Valid}(K)=\{I\subset\entk\text{ such that }K_I\text{ is a compact convex polygon}\}.\]
		For such a $I\in{\sf Valid}(K)$, we denote by $r^{(I)}_1,\ldots,r^{(I)}_{\abso{I}}$ the side-lengths of $K_I$, and $\theta^{(I)}_1,\ldots,\theta^{(I)}_{\abso{I}}$ its internal angles.
	\end{Notation}
	
	\subsection{Algorithm}
	Set $\AP^*=0$ and $I^*=\emptyset$. These quantities will keep the current best candidate for $K_\mathcal{T}$ in the memory.\\
	{\bf For all} $I\in {\sf Valid}(K)$ {\bf do}:
	
	\hspace{1cm}\begin{minipage}{.8\textwidth}%
	\textbf{Step 1:} Consider the solution to the system in the unknowns $f^{(I)}[I]$;
			\[\left\{
			\begin{array}{rll}\displaystyle
				f^{(I)}_j(f^{(I)}_j+f^{(I)}_{{j+1}})(f^{(I)}_j+f^{(I)}_{{j-1}})& = r^{(I)}_j\cdot r^{(I)}_{\wj}\cdot\sin(\theta^{(I)}_j),&\quad j\in I,\\
				f^{(I)}_j&\geq 0,&\quad j\in I.
			\end{array}
			\right.
			\]
			As said before, we know that such a solution to this system exists, but it is rarely explicit: a numerical resolution is to be considered.
			Now, the family $f^{(I)}[I]$ encodes a curve composed of parabola arcs tangent to every side of $K_I$. We therefore need to determine the tangency points.\vspace{0.1cm}
	\end{minipage}
	
	\hspace{1cm}\begin{minipage}{.8\textwidth}%
		\textbf{Step 2:} Compute
			\begin{align*}
				x^{(I)}_j=\frac{f^{(I)}_j}{f^{(I)}_{{j-1}}+f^{(I)}_j},\quad j\in I.
			\end{align*}
			\vspace{0.1cm}
	\end{minipage}
	
	\hspace{1cm}\begin{minipage}{.8\textwidth}%
		\textbf{Step 3:} Compute the set $S_I$ defined as the convex set whose boundary is composed of parabola arcs between the points determined by $x^{(I)}[I]$ on the sides of $K_I$. By construction, we have $S_I=\arg \max_{S\in\mathcal{Z}_{K_I}}\AP(S).$\vspace{0.1cm}
	\end{minipage}
	
	\hspace{1cm}\begin{minipage}{.8\textwidth}%
		\textbf{Step 4:} Check that $S_I$ is contained in $K$. If not, reject this $I$. 
			If it is, compute $\AP(S_I)$ with \eqref{eq:AP}. If $\AP^*\leq \AP(S_I)$ then we update the current best candidate for $\Dom{K}$; thus, we set \[I^*\leftarrow I \text{ and } \AP^*\leftarrow \AP(S_I).\]\vspace{0.1cm}
	\end{minipage}
	
	\textbf{Conclusion:} Among all $I$, one is realizing the convex polygon $K_I=K_\mathcal{T}$ of \Cref{thm:global}. Therefore,
	\[\AP^*(K)=\max_{I\in {\sf Valid}(K)}\AP(S_I)\]
	The final value $I^*$ we kept out of Step 3 gives $\Dom{K}=S_{I^*}.$

	\subsection*{Acknowledgements}	
	\noindent Once again, my sincere gratitude goes to Jean-François Marckert for his priceless help and guidance throughout the research and writing of this paper.
	
	\noindent Many thanks to the referees for their wise comments and careful revision of this paper, which increased both the readability of this paper, and the accuracy of the proofs.
	
	\bibliographystyle{abbrv}
	\bibliography{godlikebib.bib}
\end{document}